\documentclass{article}

\usepackage{amsmath}
\usepackage{amsthm}
\usepackage{amssymb}
\usepackage{amsfonts}

\usepackage{subeqnarray}
\usepackage{bm}
\usepackage{tikz}
\usepackage{dsfont}
\usepackage{upgreek}
\usepackage{booktabs}
\usepackage{fancyhdr}
\usepackage{multirow}
\usepackage{multicol}
\usepackage{float}
\usepackage{enumerate}
\usepackage{cases}
\usepackage{array}
\usepackage{arydshln} 

\usepackage[ruled,vlined,linesnumbered,inoutnumbered,algo2e]{algorithm2e}

\usepackage{bbding}
\usepackage{caption}
\usepackage{todonotes}

\usepackage{mathrsfs}

\usepackage{lineno}

\usepackage[top=2.5cm, bottom=2.5cm, left=3cm, right=3cm]{geometry} 

\usepackage{graphics,graphicx,epsf,epstopdf,subfigure}
\graphicspath{{./Figures/}}
\ifpdf
\DeclareGraphicsExtensions{.eps,.pdf,.png,.jpg}
\else
\DeclareGraphicsExtensions{.eps}
\fi

\usepackage[breaklinks,colorlinks,linkcolor=black,citecolor=black,urlcolor=black,hypertexnames=false]{hyperref}

\providecommand{\U}[1]{\protect\rule{.1in}{.1in}}

\newtheorem{theorem}{Theorem}[section]
\newtheorem{corollary}[theorem]{Corollary}
\newtheorem{lemma}[theorem]{Lemma}
\newtheorem{proposition}[theorem]{Proposition}
\newtheorem{definition}[theorem]{Definition}
\newtheorem{remark}{Remark}
\newtheorem{assumption}{Assumption}

\numberwithin{equation}{section}

\newcommand{\tr}{\mathrm{tr}}

\newcommand{\proj}{\mathsf{\Pi}}

\newcommand{\retr}{\mathscr{R}}
\newcommand{\grad}{\mathsf{grad}}

\newcommand{\bN}{\mathbb{N}}
\newcommand{\bR}{\mathbb{R}}

\newcommand{\cM}{\mathcal{M}}

\newcommand{\cT}{\mathcal{T}}

\newcommand{\cW}{\mathcal{W}}

\newcommand{\sfD}{\mathsf{D}}

\newcommand{\Rn}{\mathbb{R}^{n}}
\newcommand{\Rm}{\mathbb{R}^{m}}

\newcommand{\Rnn}{\mathbb{R}^{n \times n}}  
\newcommand{\Rnm}{\mathbb{R}^{n \times m}}
\newcommand{\Rmn}{\mathbb{R}^{m \times n}}

\newcommand{\Rns}{\mathbb{R}^{n \times s}}

\newcommand{\Ons}{\mathcal{O}^{n,s}}
\newcommand{\Onn}{\mathcal{O}^{n}}

\newcommand{\Opns}{\mathcal{O}^{n, s}_{+}}

\newcommand{\Sphn}{\mathcal{S}^{n - 1}}

\newcommand{\zz}{^{\top}}

\newcommand{\ff}{_{\mathrm{F}}}
\newcommand{\fs}{^2_{\mathrm{F}}}

\newcommand{\uast}{^{\ast}}

\newcommand{\uprime}{^{\prime}}

\newcommand{\lsfR}{_{\mathsf{R}}}

\newcommand{\dkh}[1]{\left(#1\right)}
\newcommand{\hkh}[1]{\left\{#1\right\}}
\newcommand{\fkh}[1]{\left[#1\right]}
\newcommand{\jkh}[1]{\left\langle#1\right\rangle}
\newcommand{\norm}[1]{\left\|#1\right\|}
\newcommand{\abs}[1]{\left\lvert #1\right\rvert}

\allowdisplaybreaks
\graphicspath{ {./results/} }
\definecolor{Gray}{rgb}{0.5,0.5,0.5}
\DeclareMathOperator*{\argmin}{arg\,min}

\makeatletter

\newcommand{\Rmnum}[1]{\expandafter\@slowromancap\romannumeral #1@}
\makeatother

\newcommand*\patchAmsMathEnvironmentForLineno[1]{
	\expandafter\let\csname old#1\expandafter\endcsname\csname#1\endcsname
	\expandafter\let\csname oldend#1\expandafter\endcsname\csname end#1\endcsname
	\renewenvironment{#1}
	{\linenomath\csname old#1\endcsname}
	{\csname oldend#1\endcsname\endlinenomath}
}
\newcommand*\patchBothAmsMathEnvironmentsForLineno[1]{
	\patchAmsMathEnvironmentForLineno{#1}
	\patchAmsMathEnvironmentForLineno{#1*}
}
\AtBeginDocument{
	\patchBothAmsMathEnvironmentsForLineno{equation}
	\patchBothAmsMathEnvironmentsForLineno{align}
	\patchBothAmsMathEnvironmentsForLineno{flalign}
	\patchBothAmsMathEnvironmentsForLineno{alignat}
	\patchBothAmsMathEnvironmentsForLineno{gather}
	\patchBothAmsMathEnvironmentsForLineno{multline}
	\patchAmsMathEnvironmentForLineno{aligned}
}

\title{An Adaptive Smoothing Algorithm for Non-Lipschitz Optimization on Manifolds with Complexity Guarantees\thanks{This work is supported by RGC grant PolyU15300024, JLFS/P-501/24, Croucher Funding Scheme for CAS AMSS-PolyU Joint Laboratory in Applied Mathematics.}}

\author{
	Lei Wang\footnote{Department of Applied Mathematics, The Hong Kong Polytechnic University, Hung Hom, Kowloon, Hong Kong, China (\href{mailto:wlkings@lsec.cc.ac.cn}{wlkings@lsec.cc.ac.cn}).}
	\and
	Xiaojun Chen\footnote{Department of Applied Mathematics, The Hong Kong Polytechnic University, Hung Hom, Kowloon, Hong Kong, China (\href{mailto:maxjchen@polyu.edu.hk}{maxjchen@polyu.edu.hk}).}
}

\date{}

\begin{document}

\maketitle

\begin{abstract}
	We study a class of optimization problems on Riemannian manifolds, where the objective function consists of a smooth term and quasi-norm type penalties with exponent $p \in (0, 1]$. The essential difficulty lies in the fact that the objective function may not be locally Lipschitz continuous, which places this type of problems beyond the reach of existing  Riemannian techniques. To overcome this obstacle, this paper constructs a general smoothing framework and establishes fundamental properties for developing efficient algorithms. In particular, we propose a smoothing Riemannian gradient algorithm equipped with a smoothing-aware AdaGrad-type stepsize rule. Its global convergence is demonstrated together with an iteration complexity of $O (\epsilon^{p - 4})$, which includes the best available iteration complexity of $O (\epsilon^{- 3})$ for Lipschitz problems with $p = 1$ as a special case. To the best of our knowledge, this is the first complexity result for non-Lipschitz optimization on Riemannian manifolds. Preliminary numerical experiments corroborate the practical efficiency of the proposed approach in real-world applications arsing from machine learning and data science. 
\end{abstract}


\section{Introduction}

In this paper, we focus on the following Riemannian optimization problem with possibly non-Lipschitz singularities,
\begin{equation}
	\label{opt:rnlip}
	\min_{X \in \cM} \hspace{2mm} f (X) := g (X) + \sum_{i = 1}^{m} \fkh{ \varphi (h_i (X)) }^{p},
\end{equation}
where $\cM$ is a compact embedded submanifold in $\Rns$, $g: \Rns \to \bR$, $h_i: \Rns \to \bR$, and $\varphi: \bR \to \bR_{+}$ are three functions, and $p \in (0, 1]$ is a constant.
Throughout this paper, we make the following blanket assumptions on problem \eqref{opt:rnlip}.

\begin{assumption}
	\label{asp:function}
	
	\mbox{}
	
	\begin{enumerate}[(i)]
		
		\item The functions $g: \Rns \to \bR$ and $h_i: \Rns \to \bR$ are continuously differentiable.
		Moreover, their Euclidean gradients $\nabla g$ and $\nabla h_i$ are Lipschitz continuous over $\cM$ with the corresponding Lipschitz constants $L_g \geq 0$ and $L_h \geq 0$, respectively.
		
		\item The function $\varphi: \bR \to \bR_{+}$ is locally Lipschitz continuous.
		
	\end{enumerate}
\end{assumption}

When $p = 1$, the objective function of problem~\eqref{opt:rnlip} is locally Lipschitz continuous.
Problems of this kind have been extensively studied in the existing literature.
The picture changes fundamentally, however, when $p \in (0, 1)$.
In this regime, the objective function is no longer locally Lipschitz continuous, thereby posing a significant challenge to solve problem~\eqref{opt:rnlip}.
Nevertheless, non-Lipschitz regularizers are often more effective in practice at promoting desirable structural properties, particularly sparsity \cite{Bian2013worst,Chen2010lower,Zhang2024riemannian}.
Some applications will be introduced in the next subsection.
The present work is devoted to developing a unified algorithmic framework for the entire range $p \in (0, 1]$, while also establishing its complexity guarantees.

It is notable that the algorithm proposed in this paper can be seamlessly extended to address the following broader class of problems,
\begin{equation*}
	\min_{X \in \cM} \hspace{2mm} f (X) := g (X) + \sum_{i = 1}^{m} \fkh{ \varphi_i (h_i (X)) }^{p_i},
\end{equation*}
where $g$ and $h_i$ satisfy Assumption~\ref{asp:function}, $\varphi_i: \bR \to \bR_{+}$ is a locally Lipschitz continuous function, and $p_i \in (0, 1]$ is a constant.
The theoretical results established in the sequel remain valid for $p = \min_{i \in \{1, 2, \dotsc, m\}} p_i$.
For the sake of clarity, we restrict our attention to the formulation of problem~\eqref{opt:rnlip}.

\subsection{Broad Applications}

Problems of the form \eqref{opt:rnlip} arise from many practical applications in machine learning and data science.
Below we briefly describe several representative instances.

\paragraph{Sparse Principal Component Analysis (SPCA).}
To enhance the interpretability of principal components, SPCA introduces the sparsity structure into the dimension-reduction process \cite{Wang2025decentralized,Wang2025distributionally}.
This problem serves as a critical procedure or preprocessing step in numerous statistical learning tasks \cite{Chen2016augmented,Wang2023smoothing}.
Let $\Ons := \{X \in \Rns \mid X\zz X = I_s\}$ be the Stiefel manifold \cite{Wang2022decentralized,Wang2024seeking} in $\Rns$ with $1 \leq s < n$.
Given a data matrix $A \in \Rnm$, SPCA identifies an orthogonal transformation by solving the optimization problem below,
\begin{equation}
	\label{opt:spca}
	\min_{X \in \Ons} \hspace{2mm} - \dfrac{1}{2} \tr \dkh{X\zz A A\zz X} + \lambda \norm{X}_1,
\end{equation}
where the $\ell_1$-norm of $X$ is defined as $\norm{X}_1 = \sum_{i = 1}^{n} \sum_{j = 1}^{s} |[X]_{i, j}|$ with $[X]_{i, j}$ being the $(i, j)$-th entry of $X$, and $\lambda > 0$ is a constant used to control the amount of sparseness.
It is readily seen that model~\eqref{opt:spca} falls within the scope of problem~\eqref{opt:rnlip} upon choosing $\varphi (t) = \abs{t}$ and $p = 1$.

\paragraph{Dual Principal Component Pursuit (DPCP).}
This problem aims to seek the sparsest vector in an $n$-dimensional linear subspace $\cW \subseteq \Rm$ with $m > n$, which finds interesting applications in robust subspace recovery \cite{Tsakiris2018dual}.
Let $E \in \Rmn$ be a matrix whose columns form an orthonormal basis of $\cW$.
Motivated by the fact that the $\ell_p$ quasi-norm ($0 < p < 1$) has great potential to promote sparsity, Zhang et al. \cite{Zhang2024riemannian} propose the following optimization model for the DPCP problem,
\begin{equation}
	\label{opt:dpcp}
	\min_{x \in \Sphn} \hspace{2mm}
	\norm{E x}_p^p,
\end{equation}
where $\Sphn = \{x \in \Rn \mid x\zz x = 1\}$ is the unit sphere in $\Rn$, and the $\ell_p$ quasi-norm of a vector $z \in \Rm$ is defined as $\norm{z}_p^p = \sum_{i = 1}^{m} \abs{[z]_i}^p$ with $[z]_i$ being the $i$-th entry of $z$.
It is evident that model~\eqref{opt:dpcp} constitutes a particular example of problem~\eqref{opt:rnlip} with $\varphi (t) = \abs{t}$ and $p \in (0, 1)$.

\paragraph{Sparse Dictionary Learning (SDL).}
As a classical representation learning paradigm, SDL has been extensively studied and widely applied in signal and image processing \cite{Zhai2020complete}, owing to its powerful capability to uncover low-dimensional latent structures embedded in high-dimensional data.
Given a collection of samples $Y = [Y_1, Y_2, \dotsc, Y_m] \in \Rnm$ with each column $Y_i \in \Rn$ being an individual sample, SDL is concerned with identifying an orthogonal dictionary matrix $X \in \Rnn$ and a sparse coefficient matrix $S \in \Rnm$ such that the observed data admits an approximate factorization $Y \approx X S$.
To tackle this problem, Zhang et al. \cite{Zhang2024riemannian} formulate an $\ell_p$ minimization model $(0 < p < 1)$ of the form
\begin{equation}
	\label{opt:sdl}
	\min_{X \in \Onn} \hspace{2mm}
	\dfrac{1}{m} \sum_{i = 1}^{m} \norm{Y_i\zz X}_p^p,
\end{equation}
where $\Onn := \{X \in \Rnn \mid X\zz X = X X\zz = I_n\}$ denotes the orthogonal group \cite{Absil2008optimization,Boumal2023introduction} and $I_n$ is the $n \times n$ identity matrix.
The above model naturally fits into the framework of problem~\eqref{opt:rnlip} as a special case, once $\varphi$ is specified as $\varphi (t) = \abs{t}$ with $p \in (0, 1)$.

\paragraph{Optimization with Nonnegative and Orthogonal Constraints.}
We consider the following optimization problem with nonnegative and orthogonal constraints \cite{Jiang2023exact,Qian2024error,Wang2025support},
\begin{equation}
	\label{opt:north}
	\min_{X \in \Opns} \hspace{2mm} g (X).
\end{equation}
where $\Opns := \{ X \in \Ons \mid [X]_{i, j} \geq 0 \mbox{ for all } i \mbox{ and } j\}$, and $g: \Rns \to \bR$ is a continuously differentiable function.
This class of problems arises ubiquitously in applications including clustering analysis and feature selection \cite{Chen2026orthogonal,Li2026unsupervised}.
Based on the global error bound of $\Opns$ derived in \cite{Chen2025tight}, we can construct an exact penalty model for problem \eqref{opt:north} as follows,
\begin{equation}
	\label{opt:north-ep}
	\min_{X \in \Ons} \hspace{2mm}
	g (X) + \alpha \sum_{i = 1}^{n} \sum_{j = 1}^{s} \fkh{\max \hkh{- [X]_{i, j}, 0}}^p,
\end{equation}
where $\alpha > 0$ is a penalty parameter and $p \in (0, 1/2]$.
According to \cite[Theorem 5.1]{Chen2025tight}, it follows that problems \eqref{opt:north} and \eqref{opt:north-ep} share the same global minimizers if $\alpha > 5 Q s^{3 / 4}$, where $Q := \max_{X \in \Ons} \norm{\nabla g (X)}\ff$.
Obviously, model~\eqref{opt:north-ep} is a special instance of problem~\eqref{opt:rnlip} by noting that $\varphi (t) = \max \{t, 0\}$ and $p \in (0, 1/2]$.

\subsection{Literature Review}

Riemannian optimization has by now evolved into a compelling and powerful paradigm for treating structured problems under geometric constraints \cite{Absil2008optimization,Boumal2023introduction}.
In recent years, nonsmooth optimization on manifolds has garnered considerable attention due to its wide applications, leading to the development of a broad spectrum of algorithms.
For a class of composite optimization problems on the Stiefel manifold, Chen et al. \cite{Chen2020proximal} introduce the manifold proximal gradient (ManPG) algorithm, which lays the foundation for subsequent advancements \cite{Chen2024nonsmooth}.
Later on, Huang and Wei \cite{Huang2022riemannian} extend the framework of ManPG to general Riemannian manifolds.
It is noteworthy that both algorithms require an exact solution to a subproblem in the tangent space at each iteration.
Their inexact variants are proposed and analyzed in \cite{Huang2023inexact,Jiang2025inexact,Zheng2025new}.
Additionally, Li et al. \cite{Li2021weakly} study the convergence rate of Riemannian subgradient algorithms for minimizing weakly convex functions over the Stiefel manifold.
Some efforts are also devoted to handling nonsmooth optimization problems on manifolds via smoothing techniques, including \cite{Beck2023dynamic,Peng2023riemannian,Wang2025distributionally}.
Moreover, there is a growing body of researches on Riemannian primal-dual frameworks, especially in the form of augmented Lagrangian methods \cite{Deng2025oracle} and ADMM-type methods \cite{Li2025riemannian}.
It deserves particular emphasis that the aforementioned algorithms are predominantly developed under the assumption that the objective function is locally Lipschitz continuous.
Consequently, they are only applicable to problem~\eqref{opt:rnlip} when $p = 1$.

The investigation of problem~\eqref{opt:rnlip} in the case $p \in (0, 1)$ is still in its infancy.
The Riemannian smoothing steepest descent (RSSD) algorithm proposed in \cite{Zhang2024riemannian} proceeds by applying a Riemannian gradient descent scheme to a smooth approximation of the objective function, with the stepsize determined by a line-search procedure.
Although its convergence has been justified, the iteration complexity of RSSD remains largely unexplored.
Beyond this, the line-search mechanism adopted in RSSD incurs a nonnegligible computational burden, as each backtracking step calls for an evaluation of the retraction operator.
Such a limitation significantly impairs the practical performance of RSSD when confronted with large-scale applications.

In contrast to the Riemannian setting, the study of non-Lipschitz optimization in Euclidean spaces is far more extensive.
For example, Bian and Chen \cite{Bian2013worst} propose a smoothing quadratic regularization method and establish worst-case complexity bounds for a class of unconstrained non-Lipschitz optimization problems.
The complexity analysis is further extended in \cite{Bian2015complexity} to the interior point algorithm for non-Lipschitz optimization with box constraints.
For more structured models, Liu et al. \cite{Liu2016smoothing} develop a smoothing sequential quadratic programming framework for composite $L_p (0 < p < 1)$ minimization problems over polyhedra.
Moreover, the complexity results are established in \cite{Chen2019complexity} for partially separable convexly constrained optimization with non-Lipschitz singularities.
This line of work is pushed further by Chen and Toint \cite{Chen2021high}, who derive high-order evaluation complexity bounds for convexly constrained problems involving non-Lipschitz group sparsity terms.
In summary, even in the Euclidean setting, a tractable complexity analysis can be carried out only when the functions $\{h_i\}_{i = 1}^{m}$ possess certain special structural properties, such as the fully separable structure explored in \cite{Bian2013worst,Bian2015complexity} and the partially separable structure considered in \cite{Chen2021high,Chen2019complexity,Liu2016smoothing}.
These limitations underscore the intrinsic obstacle to establishing the iteration complexity of algorithms aimed at solving problem~\eqref{opt:rnlip}.

\subsection{Contribution}

The contributions of this paper can be summarized in the following three aspects.

(i)
Based on the proposed smoothing paradigm, we design an adaptive smoothing Riemannian gradient algorithm (ASRGA) for problem~\eqref{opt:rnlip}.
This approach couples a carefully scheduled decay of the smoothing parameter with Riemannian gradient steps on the smoothed objective function.
To avoid the practical burden of estimating problem parameters or performing line-search procedures, we endow ASRGA with an adaptive stepsize policy, inspired by the classical AdaGrad algorithm \cite{Duchi2011adaptive,Gratton2024complexity}.
This new update rule can be viewed as a smoothing-aware AdaGrad scheme rather than a direct transplantation of AdaGrad.
In ASRGA, both the gradient accumulation and the effective stepsize are scaled in accordance with the current smoothing level to compensate for the blow-up of Lipschitz constants.

(ii)
We establish the global convergence of ASRGA to a stationary point of problem~\eqref{opt:rnlip}.
Moreover, an $\epsilon$-approximate stationary point can be found by ASRGA within $O (\epsilon^{p -4})$ iterations.
ASRGA achieves these results in an entirely parameter-free manner, thereby reconciling practical implementability with rigorous worst-case guarantees.
Table~\ref{tb:complexity} summarizes the complexity results of several representative algorithms for solving nonsmooth optimization problems on manifolds.
When $p = 1$, the iteration complexity of ASRGA specializes to $O (\epsilon^{-3})$, which is consistent with existing results for Riemannian optimization with Lipschitz continuous functions.
Moreover, our complexity analysis reveals that smaller values of $p$ render problem~\eqref{opt:rnlip} increasingly challenging to solve.
These findings bridge a notable gap between Lipschitz and non-Lipschitz optimization on manifolds.
To the best of our knowledge, the present work delivers the first complexity bound for problem~\eqref{opt:rnlip}.

(iii)
In numerical experiments, ASRGA consistently enjoys a pronounced computational advantage over existing methods, which lends empirical support to the effectiveness and robustness of the proposed smoothing methodology.
This superiority persists even in the special case of $p = 1$, where our algorithm continues to outperform DSGM \cite{Beck2023dynamic} and ManPG \cite{Chen2020proximal}.

\begin{table}[ht]
	\setlength{\extrarowheight}{1.2pt}
	\centering
	\begin{minipage}{\textwidth}
		\caption{Complexity results of some representative algorithms for nonsmooth optimization problems on manifolds.}
		\label{tb:complexity}
		\begin{tabular*}{\textwidth}{@{\hspace{5pt}\extracolsep{\fill}}cccc@{\hspace{5pt}}}
			\toprule
			Algorithm & Mechanism & Objective & Complexity \\
			\midrule
			RSM \cite{Li2021weakly} & subgradient & Lipschitz continuous & $O (\epsilon^{-4})$ \\
			\hdashline[0.5pt/2pt]
			RADMM \cite{Li2025riemannian} & ADMM & Lipschitz continuous & $O (\epsilon^{-4})$ \\
			\hdashline[0.5pt/2pt]
			iRPDC \cite{Jiang2025inexact} & proximal gradient & Lipschitz continuous & $O (\epsilon^{-3})$ \\
			\hdashline[0.5pt/2pt]
			ManIAL \cite{Deng2025oracle} & augmented Lagrangian & Lipschitz continuous & $O (\epsilon^{-3})$ \\
			\hdashline[0.5pt/2pt]
			DSGM \cite{Beck2023dynamic} & smoothing & Lipschitz continuous & $O (\epsilon^{-3})$ \\
			\hdashline[0.5pt/2pt]
			RSSD \cite{Zhang2024riemannian} & smoothing & non-Lipschitz & unknown \\
			\hdashline[0.5pt/2pt]
			ASRGA [this work] & smoothing & non-Lipschitz & $O (\epsilon^{p - 4})$ \\
			\bottomrule
		\end{tabular*}
	\end{minipage}
\end{table}

\subsection{Organization}

The rest of this paper proceeds as follows.
Section~\ref{sec:preliminary} draws into some preliminaries of Riemannian optimization.
Section~\ref{sec:smoothing} is dedicated to proposing an effective smoothing scheme that approximates the objective function of problem~\eqref{opt:rnlip}, which is followed by the development of an adaptive smoothing Riemannian gradient algorithm in Section~\ref{sec:algorithm}.
Then we establish the global convergence and the iteration complexity of the proposed algorithm in Section~\ref{sec:convergence}.
Numerical results are presented in Section~\ref{sec:numerical} to illustrate the superior computational efficiency of our algorithm relative to existing methods.
Finally, this paper concludes with key insights and directions for future research in Section~\ref{sec:conclusion}.

\section{Preliminaries} 

\label{sec:preliminary}

In this section, we introduce the notations used throughout this paper and provide a concise overview of several fundamental concepts in Riemannian optimization.

\subsection{Basic Notations}

We use $\bR$ and $\bN$ to denote the sets of real and natural numbers, respectively.
The set of nonnegative real numbers is represented by $\bR_{+}$.
The Euclidean inner product of two matrices $C_1, C_2$ with the same size is defined as $\jkh{C_1, C_2} = \tr (C_1\zz C_2)$, where $\tr (B)$ stands for the trace of a square matrix $B$.
The Frobenius norm and the $\ell_q$ norm with $q \geq 1$ of a given matrix $C$ are denoted by $\norm{C}\ff$ and $\norm{C}_q$, respectively.

\subsection{Riemannian Gradients}

For each point $X \in \cM$, the tangent space to $\cM$ at $X$ is referred to as $\cT_{X} \cM$.
In this paper, we consider the Riemannian metric $\jkh{\cdot, \cdot}_X$ on $\cT_{X} \cM$ that is induced from the Euclidean inner product $\jkh{\cdot, \cdot}$, i.e., $\jkh{D_1, D_2}_X = \jkh{D_1, D_2} = \tr(D_1\zz D_2)$ for any $D_1, D_2 \in \cT_{X} \cM$.
The tangent bundle of $\cM$ is denoted by $\cT \cM = \{(X, D) \mid X \in \cM, D \in \cT_X \cM\}$, that is, the disjoint union of the tangent spaces of $\cM$.
Additionally, we use the notation $\proj_{\cT_X \cM} (\cdot)$ to represent the orthogonal projection operator onto $\cT_X \cM$.
In the context of the Stiefel manifold $\Ons$, the tangent space at $X$ can be described as $\cT_{X} \Ons = \{D \in \Rns \mid X\zz D + D\zz X = 0\}$, and its orthogonal projection operator is given by
\begin{equation*}
	\proj_{\cT_X \Ons} (D) = D - X (X\zz D + D\zz X) / 2,
\end{equation*}
for any $D \in \Rns$.

For a smooth function $f$, the Riemannian gradient at $X \in \cM$, denoted by $\grad\, f (X)$, is defined as the unique element of $\cT_X \cM$ satisfying
\begin{equation*}
	\jkh{\grad\, f (X), D} = \sfD f (X) [D], \quad \forall\, D \in \cT_X \cM,
\end{equation*}
where $\sfD f (X) [D]$ is the directional derivative of $f$ along the direction $D$ at the point $X$.
Since $f$ is defined on an embedded submanifold in the Euclidean space, its Riemannian gradient can be computed by projecting the Euclidean gradient $\nabla f (X)$ onto the tangent space as follows,
\begin{equation*}
	\grad\, f (X) = \proj_{\cT_X \cM} \dkh{\nabla f (X)}.
\end{equation*}

\subsection{Retractions}

In contrast to the Euclidean setting, the point $X + D$ does not lie in the manifold in general for $X \in \cM$ and $D \in \cT_{X} \cM$, due to the absence of a linear structure in $\cM$.
The interplay between $\cM$ and $\cT_{X} \cM$ is typically carried out via the exponential mappings, which are usually computationally intensive to evaluate in practice.
As an alternative, the concept of a retraction, which is a first-order approximation of the exponential mapping and can be more amenable to computation, is given as follows.

\begin{definition}[\cite{Absil2008optimization}] \label{def:retr}
	A retraction on a manifold $\cM$ is a smooth mapping $\retr: \cT \cM \to \cM$, and the restriction of $\retr$ to $\cT_{X} \cM$, denoted by $\retr_{X}$, satisfies the following two properties for any $X \in \cM$.
	\begin{enumerate}[(i)]
		
		\item It holds that $\retr_{X} (0_X) = X$, where $0_X$ is the zero vector in $\cT_{X} \cM$.
		
		\item The differential of $\retr_{X}$ at $0_X$ is an identity map on $\cT_{X} \cM$.
		
	\end{enumerate}
\end{definition}

By leveraging the retraction $\retr_{X} (D)$, we can obtain a point by moving away from $X \in \cM$ along the direction $D \in \cT_{X} \cM$, while remaining on the manifold.
To this extent, it defines an update rule to preserve the feasibility.
For instance, there are various practical realizations of retractions on the Stiefel manifold, including QR factorization, polar decomposition and Cayley transformation.
A useful property of the retraction is highlighted in the following lemma.

\begin{lemma}[{\cite[Lemma~2.7]{Boumal2018global}}]
	\label{le:retr}
	Let $\cM$ be a compact embedded submanifold of $\Rns$ and $\psi: \Rns \to \bR$ be a continuously differentiable function.
	Suppose that $\nabla \psi$ is Lipschitz continuous over $\cM$ with the corresponding Lipschitz constants $L_{\psi}$ and $\norm{\nabla \psi (X)}\ff \leq M_{\psi}$ for all $X \in \cM$.
	There exist two constants $R_1 > 0$ and $R_2 > 0$ independent of $\psi$ such that
	\begin{equation*}
		\psi (\retr_{X} (D))
		\leq \psi (X)
		+ \jkh{\grad\, \psi (X), D}
		+ \dfrac{R_1^2 L_{\psi} + 2 R_2 M_{\psi}}{2} \norm{D}\fs,
	\end{equation*}
	for any $X \in \cM$ and $D \in \cT_{X} \cM$.
\end{lemma}

\section{Smoothing Approximation}

\label{sec:smoothing}

To address the challenges posed by the non-Lipschitz singularities of problem~\eqref{opt:rnlip}, we propose to leverage the smoothing approximation technique \cite{Chen2012smoothing}.
The purpose of this section is to develop a suitable smoothing surrogate for the objective function of problem~\eqref{opt:rnlip} and to analyze its favorable properties.

\subsection{Smoothing Function}

The nonsmooth structure of problem~\eqref{opt:rnlip} originates from the function $\varphi$.
Accordingly, we begin by introducing a smoothing function of $\varphi$, which in turn induces a corresponding smoothing surrogate for the objective function of problem~\eqref{opt:rnlip}.

\begin{definition}
	\label{def:smoothing}
	A function $\tilde{\varphi}_{\mu}: \bR \to \bR_{+}$ with a smoothing parameter $\mu > 0$ is called a smoothing function of $\varphi: \bR \to \bR_{+}$ if it satisfies the following conditions.
	\begin{enumerate}[(i)]
		
		\item For any $\mu > 0$, $\tilde{\varphi}_{\mu}$ is continuously differentiable over $\bR$.
		
		\item There exists a constant $\sigma > 0$ such that
		\begin{equation}
			\label{eq:varphi-lb}
			\tilde{\varphi}_{\mu} (t) \geq \sigma \mu,
		\end{equation}
		for all $t \in \bR$ and $\mu > 0$.
		
		\item For any $t \in \bR$, it holds that
		\begin{equation}
			\label{eq:varphi-limit}
			\lim_{s \to t, \, \mu \downarrow 0} \tilde{\varphi}_{\mu} (s) = \varphi (t).
		\end{equation}
		
		\item There exists a constant $\kappa > 0$ such that
		\begin{equation}
			\label{eq:varphi-lip}
			\abs{\tilde{\varphi}_{\mu_1} (t) - \tilde{\varphi}_{\mu_2} (t)} \leq \kappa \abs{\mu_1 - \mu_2},
		\end{equation}
		for all $t \in \bR$ and $\mu_1, \mu_2 > 0$.
		
		\item There exists a constant $M_{\varphi} > 0$ such that
		\begin{equation}
			\label{eq:varphi-ub}
			\abs{ \tilde{\varphi}_{\mu}\uprime (t) } \leq M_{\varphi},
		\end{equation}
		for all $t \in \bR$ and $\mu > 0$.
		
		\item There exists a constant $L_{\varphi} > 0$ such that, for any $\mu > 0$, $\tilde{\varphi}_{\mu}\uprime$ is Lipschitz continuous over $\bR$ with the corresponding Lipschitz constant $L_{\varphi} \mu^{-1}$.
		
	\end{enumerate}
\end{definition}

In the special case of $p = 1$, the property \eqref{eq:varphi-lb} can be dispensed with, while the subsequent theoretical analysis remains intact by adopting the convention $0^0 = 1$.
As an illustrative example, the smoothing function of $\varphi (t) = \abs{t}$ with $\mu > 0$ can be designed as
\begin{equation}
	\label{eq:tilde-abs}
	\tilde{\varphi}_{\mu} (t) = \left\{
	\begin{aligned}
		& \abs{t}, && \mbox{if~} t > \mu \mbox{~or~} t < - \mu, \\
		& \dfrac{t^2}{2 \mu} + \dfrac{\mu}{2}, && \mbox{if~} - \mu \leq t \leq \mu,
	\end{aligned}
	\right.
\end{equation}
which fulfills all the requirements in Definition \ref{def:smoothing}.
We refer interested readers to \cite{Chen2012smoothing} for more examples of smoothing functions applied to Lipschitz continuous functions.

Building on the construction of $\tilde{\varphi}_{\mu}$, we approximate the objective function of problem~\eqref{opt:rnlip} by the following smoothing function,
\begin{equation}
	\label{eq:tilde-f}
	\tilde{f}_{\mu} (X) = g (X) + \sum_{i = 1}^{m} \fkh{ \tilde{\varphi}_{\mu} (h_i (X)) }^p,
\end{equation}
with $\mu > 0$.
As can be directly inferred from the property \eqref{eq:varphi-lb}, the function $\tilde{f}_{\mu}$ is locally Lipschitz continuous.
Moreover, $\tilde{f}_{\mu}$ is continuously differentiable and its Euclidean gradient is given by
\begin{equation*}
	\nabla \tilde{f}_{\mu} (X) = \nabla g (X) + p \sum_{i = 1}^{m} \fkh{ \tilde{\varphi}_{\mu} (h_i (X)) }^{p - 1} \tilde{\varphi}_{\mu}\uprime (h_i (X)) \nabla h_i (X).
\end{equation*}
The smoothing function $\tilde{f}_{\mu}$ underpins the design of algorithms presented later in this paper.

\begin{remark}
	The property \eqref{eq:varphi-lb} in Definition~\ref{def:smoothing} is indispensable for coping with the non-Lipschitz nature of problem~\eqref{opt:rnlip} when $p \in (0, 1)$.
	In the existing literature \cite{Beck2023dynamic,Peng2023riemannian}, the Moreau envelope is commonly invoked as a smoothing function.
	However, it does not satisfy the condition~ \eqref{eq:varphi-lb}.
	For instance, the Moreau envelope of $\varphi (t) = \abs{t}$ with a smoothing parameter $\mu > 0$ is expressed as
	\begin{equation*}
		\bar{\varphi}_{\mu} (t)
		= \min_{s \in \bR} \hkh{ \abs{s} + \dfrac{1}{2 \mu} (s - t)^2 }
		= \tilde{\varphi}_{\mu} (t) - \dfrac{\mu}{2},
	\end{equation*}
	which vanishes at $t = 0$.
	Here, $\tilde{\varphi}_{\mu}$ is the smoothing function defined in \eqref{eq:tilde-abs}.
	By using this Moreau envelope, we arrive at an approximation $\bar{\omega}_{\mu} (t) = [\bar{\varphi}_{\mu} (t)]^p$ for the non-Lipschitz function $\omega (t) = \abs{t}^p$.
	Then $\bar{\omega}_{\mu}$ fails to be differentiable whenever $p \in (0, 1/2]$.
	Even for $p \in (1/2, 1)$, while differentiability is recovered, the Lipschitz continuity of its gradient is still lacking.
	It is only in the case $p = 1$ that $\bar{\omega}_{\mu}$ becomes differentiable with a Lipschitz continuous gradient.
\end{remark}

\subsection{Useful Properties}

This subsection elucidates several desirable properties of the smoothing function $\tilde{f}_{\mu}$, which will be instrumental in the subsequent analysis.

Owing to the compactness of $\cM$, there exist two constants $M_g > 0$ and $M_h > 0$ such that $\norm{\nabla g (X)}\ff \leq M_g$ and $\norm{\nabla h_i (X)}\ff \leq M_h$ for all $X \in \cM$ and $i \in \{1, 2, \dotsc, m\}$.
The following lemma gives the Lipschitz constant of $\nabla \tilde{f}_{\mu}$.

\begin{lemma}
	\label{le:sgrad-lip}
	Let $\mu \in (0, 1]$.
	The gradient $\nabla \tilde{f}_{\mu}$ is Lipschitz continuous over $\cM$ with the Lipschitz constant $L_{\mu} := L_g + L_1 L_h \mu^{p - 1} + L_2 \mu^{p - 2}$, where $L_1 = p m \sigma^{p - 1} M_{\varphi} > 0$ and $L_2 = p m \sigma^{p - 2} M_h^2 ( \sigma L_{\varphi} + (1 - p)  M_{\varphi}^2 ) > 0$ are two constants.
\end{lemma}

\begin{proof}
	Let $H_i (X) = \fkh{ \tilde{\varphi}_{\mu} (h_i (X)) }^{p - 1} \tilde{\varphi}_{\mu}\uprime (h_i (X)) \nabla h_i (X)$ for $X \in \cM$ and $i \in \{1, 2, \dotsc, m\}$.
	Then straightforward calculations yield that
	\begin{equation}
		\label{eq:h-xy}
		\norm{H_i (X) - H_i (Y)}\ff \leq U (X, Y) + V (X, Y),
	\end{equation}
	where
	\begin{equation*}
		U (X, Y) = \norm{\fkh{ \tilde{\varphi}_{\mu} (h_i (X)) }^{p - 1} \tilde{\varphi}_{\mu}\uprime (h_i (X)) \dkh{ \nabla h_i (X) - \nabla h_i (Y) }}\ff,
	\end{equation*}
	and
	\begin{equation*}
		V (X, Y) = \norm{ \dkh{ \fkh{ \tilde{\varphi}_{\mu} (h_i (X)) }^{p - 1} \tilde{\varphi}_{\mu}\uprime (h_i (X)) - \fkh{ \tilde{\varphi}_{\mu} (h_i (Y)) }^{p - 1} \tilde{\varphi}_{\mu}\uprime (h_i (Y)) } \nabla h_i (Y) }\ff.
	\end{equation*}
	By virtue of the properties \eqref{eq:varphi-lb} and \eqref{eq:varphi-ub}, we obtain that
	\begin{equation}
		\label{eq:u-xy}
		\begin{aligned}
			U (X, Y)
			\leq {} & \dfrac{\abs{\tilde{\varphi}_{\mu}\uprime (h_i (X))}}{\abs{ \tilde{\varphi}_{\mu} (h_i (X)) }^{1 - p}} \norm{\nabla h_i (X) - \nabla h_i (Y)}\ff \\
			\leq {} & \sigma^{p - 1} \mu^{p - 1} M_{\varphi} L_h \norm{X - Y}\ff
			= \dfrac{L_1 L_h \mu^{p - 1}}{p m} \norm{X - Y}\ff.
		\end{aligned}
	\end{equation}
	Moreover, it holds that
	\begin{equation}
		\label{eq:v-xy-w}
		V (X, Y) \leq M_h W (X, Y),
	\end{equation}
	where $W (X, Y) = | \fkh{ \tilde{\varphi}_{\mu} (h_i (X)) }^{p - 1} \tilde{\varphi}_{\mu}\uprime (h_i (X)) - \fkh{ \tilde{\varphi}_{\mu} (h_i (Y)) }^{p - 1} \tilde{\varphi}_{\mu}\uprime (h_i (Y)) |$.
	According to the triangle inequality, it follows that
	\begin{equation}
		\label{eq:w-xy}
		\begin{aligned}
			W (X, Y)
			= {} & \abs{ \fkh{ \tilde{\varphi}_{\mu} (h_i (X)) }^{p - 1} \tilde{\varphi}_{\mu}\uprime (h_i (X)) - \fkh{ \tilde{\varphi}_{\mu} (h_i (Y)) }^{p - 1} \tilde{\varphi}_{\mu}\uprime (h_i (Y)) } \\
			\leq {} & \abs{ \fkh{ \tilde{\varphi}_{\mu} (h_i (X)) }^{p - 1} } \abs{ \tilde{\varphi}_{\mu}\uprime (h_i (X)) - \tilde{\varphi}_{\mu}\uprime (h_i (Y)) } \\
			& + \abs{ \fkh{ \tilde{\varphi}_{\mu} (h_i (X)) }^{p - 1} - \fkh{ \tilde{\varphi}_{\mu} (h_i (Y)) }^{p - 1} } \abs{ \tilde{\varphi}_{\mu}\uprime (h_i (Y)) } \\
			\leq {} & \sigma^{p - 1} \mu^{p - 1} \abs{ \tilde{\varphi}_{\mu}\uprime (h_i (X)) - \tilde{\varphi}_{\mu}\uprime (h_i (Y)) } \\
			& + M_{\varphi} \abs{ \fkh{ \tilde{\varphi}_{\mu} (h_i (X)) }^{p - 1} - \fkh{ \tilde{\varphi}_{\mu} (h_i (Y)) }^{p - 1} }.
		\end{aligned}
	\end{equation}
	Since $\tilde{\varphi}_{\mu}\uprime$ is Lipschitz continuous with the corresponding Lipschitz constant $L_{\varphi} \mu^{-1}$, we have
	\begin{equation}
		\label{eq:w-xy-1}
		\abs{ \tilde{\varphi}_{\mu}\uprime (h_i (X)) - \tilde{\varphi}_{\mu}\uprime (h_i (Y)) }
		\leq \dfrac{L_{\varphi}}{\mu} \abs{h_i (X) - h_i (Y)}
		\leq \dfrac{L_{\varphi} M_h}{\mu} \norm{X - Y}\ff.
	\end{equation}
	And a straightforward verification reveals that
	\begin{equation}
		\label{eq:w-xy-2}
		\begin{aligned}
			\abs{ \fkh{ \tilde{\varphi}_{\mu} (h_i (X)) }^{p - 1} - \fkh{ \tilde{\varphi}_{\mu} (h_i (Y)) }^{p - 1} }
			\leq {} & (1 - p) \sigma^{p - 2} \mu^{p - 2} \abs{\tilde{\varphi}_{\mu} (h_i (X)) - \tilde{\varphi}_{\mu} (h_i (Y))} \\
			\leq {} & (1 - p) \sigma^{p - 2} \mu^{p - 2} M_{\varphi} M_h \norm{X - Y}\ff.
		\end{aligned}
	\end{equation}
	By combining three relationships \eqref{eq:w-xy}, \eqref{eq:w-xy-1}, and \eqref{eq:w-xy-2}, we arrive at
	\begin{equation*}
		W (X, Y)
		\leq \sigma^{p - 2} M_h \dkh{ \sigma L_{\varphi} + (1 - p)  M_{\varphi}^2  } \mu^{p - 2} \norm{X - Y}\ff
		= \dfrac{L_2 \mu^{p - 2}}{p m M_h} \norm{X - Y}\ff,
	\end{equation*}
	which together with \eqref{eq:v-xy-w} implies that
	\begin{equation}
		\label{eq:v-xy}
		V (X, Y) \leq \dfrac{L_2 \mu^{p - 2}}{p m} \norm{X - Y}\ff.
	\end{equation}
	From \eqref{eq:h-xy}, \eqref{eq:u-xy}, and \eqref{eq:v-xy}, it can be readily verified that
	\begin{equation*}
		\norm{H_i (X) - H_i (Y)}\ff \leq \dfrac{1}{p m} \dkh{ L_1 L_h \mu^{p - 1} + L_2 \mu^{p - 2} } \norm{X - Y}\ff.
	\end{equation*}
	Hence, we obtain that
	\begin{equation*}
		\begin{aligned}
			\norm{ \nabla \tilde{f}_{\mu} (X) - \nabla \tilde{f}_{\mu} (Y) }\ff
			\leq {} & \norm{\nabla g (X) - \nabla g (Y)}\ff
			+ p \sum_{i = 1}^{m} \norm{H_i (X) - H_i (Y)}\ff \\
			\leq {} & \dkh{ L_g + L_1 L_h \mu^{p - 1} + L_2 \mu^{p - 2} } \norm{X - Y}\ff,
		\end{aligned}
	\end{equation*}
	for all $X \in \cM$ and $Y \in \cM$.
	The proof is completed.
\end{proof}

In the following lemma, we proceed to show that the Euclidean gradient of the smoothing function $\tilde{f}_{\mu}$ is uniformly bounded on the manifold $\cM$.

\begin{lemma}
	There exists a positive constant $M_{\mu} := M_g + L_1 M_h \mu^{p - 1}$ such that $\|\nabla \tilde{f}_{\mu} (X)\|\ff \leq M_{\mu}$ for any $X \in \cM$ and $\mu \in (0, 1]$.
\end{lemma}

\begin{proof}
	We continue to use the notation $H_i (X) = \fkh{ \tilde{\varphi}_{\mu} (h_i (X)) }^{p - 1} \tilde{\varphi}_{\mu}\uprime (h_i (X)) \nabla h_i (X)$ for $X \in \cM$ and $i \in \{1, 2, \dotsc, m\}$.
	Then it follows from the relationships \eqref{eq:varphi-lb} and \eqref{eq:varphi-ub} that
	\begin{equation*}
		\norm{H_i (X)}\ff
		\leq \dfrac{\abs{\tilde{\varphi}_{\mu}\uprime (h_i (X))}}{\abs{ \tilde{\varphi}_{\mu} (h_i (X)) }^{1 - p}} \norm{\nabla h_i (X)}\ff
		\leq \sigma^{p - 1} \mu^{p - 1} M_{\varphi} M_h
		= \dfrac{L_1 M_h \mu^{p - 1}}{p m}.
	\end{equation*}
	Hence, we obtain that
	\begin{equation*}
		\norm{\nabla \tilde{f}_{\mu} (X)}\ff
		\leq \norm{\nabla g (X)}\ff
		+ p \sum_{i = 1}^{m} \norm{H_i (X)}\ff
		\leq M_g + L_1 M_h \mu^{p - 1},
	\end{equation*}
	which completes the proof.
\end{proof}

Based on the two lemmas above, we can prove that the composition of $\tilde{f}_{\mu}$ and a retraction also satisfies a certain Lipschitz smoothness property.

\begin{corollary}
	\label{coro:retr}
	Let $\retr$ be a retraction on $\cM$.
	Then, for any $X \in \cM$, $D \in \cT_{X} \cM$, and $\mu \in (0, 1]$, we have
	\begin{equation*}
		\tilde{f}_{\mu} (\retr_{X} (D))
		\leq \tilde{f}_{\mu} (X) + \jkh{\grad\, \tilde{f}_{\mu} (X), D} + \dfrac{R_1^2 L_{\mu} + 2 R_2 M_{\mu}}{2} \norm{D}\fs,
	\end{equation*}
	where $R_1$ and $R_2$ are two constants in Lemma~\ref{le:retr}.
\end{corollary}

\begin{proof}
	This is a direct consequence of Lemma~\ref{le:retr} and the proof is therefore omitted here.
\end{proof}

Another noteworthy result is that the smoothing function $\tilde{f}_{\mu}$ enjoys uniform Lipschitz continuity with respect to $\mu$.

\begin{lemma}
	\label{le:sgrad-mu}
	For all $X \in \cM$ and $0 < \mu_2 \leq \mu_1 \leq 1$, it holds that
	\begin{equation*}
		\abs{\tilde{f}_{\mu_1} (X) - \tilde{f}_{\mu_2} (X)} \leq m \kappa \sigma^{p - 1} \dkh{\dfrac{\mu_1}{\mu_2}}^{1 - p} \dkh{\mu_1^p - \mu_2^p}.
	\end{equation*}
\end{lemma}

\begin{proof}
	According to the mean value theorem, we have
	\begin{equation}
		\label{eq:varphi-p}
		\begin{aligned}
			\abs{\fkh{ \tilde{\varphi}_{\mu_1} (h_i (X)) }^p - \fkh{ \tilde{\varphi}_{\mu_2} (h_i (X)) }^p}
			\leq {} & p \sigma^{p - 1} \mu_2^{p - 1} \abs{\tilde{\varphi}_{\mu_1} (h_i (X)) - \tilde{\varphi}_{\mu_2} (h_i (X))} \\
			\leq {} & p \kappa \sigma^{p - 1} \mu_2^{p - 1} \abs{\mu_1 - \mu_2},
		\end{aligned}
	\end{equation}
	where the last inequality follows from \eqref{eq:varphi-lip}.
	Applying the mean value theorem again leads to that
	\begin{equation}
		\label{eq:mu-p}
		\abs{\mu_1^p - \mu_2^p}
		\geq p \mu_1^{p - 1} \abs{\mu_1 - \mu_2}.
	\end{equation}
	Hence, in view of two relationships \eqref{eq:varphi-p} and \eqref{eq:mu-p}, we arrive at
	\begin{equation}
		\label{eq:varphi-mu-i}
		\abs{\fkh{ \tilde{\varphi}_{\mu_1} (h_i (X)) }^p - \fkh{ \tilde{\varphi}_{\mu_2} (h_i (X)) }^p}
		\leq \kappa \sigma^{p - 1} \dkh{ \dfrac{\mu_1}{\mu_2} }^{1 - p} \abs{\mu_1^p - \mu_2^p}.
	\end{equation}
	The last thing to do in the proof is to combine the inequality \eqref{eq:varphi-mu-i} with
	\begin{equation*}
		\abs{\tilde{f}_{\mu_1} (X) - \tilde{f}_{\mu_2} (X)} \leq \sum_{i = 1}^{m} \abs{\fkh{ \tilde{\varphi}_{\mu_1} (h_i (X)) }^p - \fkh{ \tilde{\varphi}_{\mu_2} (h_i (X)) }^p}.
	\end{equation*}
	We complete the proof.
\end{proof}

We end this subsection by providing both lower and upper bounds for the smoothing function over the compact manifold $\cM$.

\begin{lemma}
	\label{le:bound}
	There exist two constants $\underline{f}$ and $\overline{f}$ such that
	\begin{equation}
		\label{eq:bound}
		\underline{f}
		\leq \tilde{f}_{\mu} (X)
		\leq \overline{f},
	\end{equation}
	for any $X \in \cM$ and $\mu \in (0, 1]$.
\end{lemma}

\begin{proof}
	This result follows immediately from the continuity of $\tilde{f}_{\mu}$ and the compactness of $\cM$.
	The proof is thus omitted.
\end{proof}

\subsection{Stationarity Condition}

The algorithm proposed in this paper is based on the smoothing approximation technique.
Thus, it is natural that the convergence result is closely tied to the specific smoothing function employed.

Below is the definition of the Riemannian subdifferential associated with the smoothing function, which serves as a foundational concept in our analysis.

\begin{definition}[\cite{Zhang2024riemannian}]
	\label{def:rsubd}
	The Riemannian subdifferential of $f$ associated with the smoothing function $\tilde{f}_{\mu}$ at $X \in \cM$ is defined as
	\begin{equation*}
		\tilde{\partial}\lsfR f (X) = \hkh{ D \in \cT_X \cM \mid \grad\, \tilde{f}_{\mu} (Z) \to D, \cM \ni Z \to X, \mu \downarrow 0 }.
	\end{equation*}
\end{definition}

In light of \cite[Theorem 1]{Zhang2024riemannian}, a necessary condition for a point $X \in \cM$ to be a local minimizer of problem~\eqref{opt:rnlip} is that $0 \in \tilde{\partial}\lsfR f (X)$.
This conclusion naturally motivates the definition of a stationary point given below.

\begin{definition}
	A point $X \in \cM$ is called a stationary point of problem~\eqref{opt:rnlip} if it satisfies the inclusion $0 \in \tilde{\partial}\lsfR f (X)$, namely,
	\begin{equation}
		\label{eq:stationary}
		\liminf_{\cM \ni Z \to X, \, \mu \downarrow 0} \norm{\grad\, \tilde{f}_{\mu} (Z)}\ff = 0.
	\end{equation}
\end{definition}

Building upon the preceding definition, we introduce the concept of an approximate first-order stationary point, which arises from a perturbation of the stationarity condition \eqref{eq:stationary}.

\begin{definition}
	\label{def:stationary}
	A point $X \in \cM$ is called an $\epsilon$-approximate stationary point of problem~\eqref{opt:rnlip} if there exists a smoothing parameter $\mu \in (0, \epsilon]$ such that
	\begin{equation}
		\label{eq:app-sta}
		\norm{\grad\, \tilde{f}_{\mu} (X)}\ff \leq \epsilon.
	\end{equation}
\end{definition}

The requirement~\eqref{eq:app-sta} can also be interpreted as an approximate stationarity condition of a smoothing formulation that is sufficiently close to problem~\eqref{opt:rnlip}.
Similar definitions have been widely adopted in the literature \cite{Garmanjani2013smoothing,Zhang2024riemannian}.
In the remainder of this paper, we devote our attention to the development of a smoothing algorithm for problem~\eqref{opt:rnlip} and the analysis of its iteration complexity in reaching an $\epsilon$-approximate stationary point.

\begin{remark}
	When $p = 1$, we can formulate a smoothing function based on the Moreau envelope.
	Under this setting, the notion of an $\epsilon$-approximate stationary point in Definition~\ref{def:stationary} is consistent with that introduced in \cite{Beck2023dynamic}.
	As a result, the iteration complexity derived in the sequel naturally specializes to this case.
\end{remark}

\begin{remark}
	In the majority of existing works, an $\epsilon$-approximate stationary point for non-Lipschitz optimization problems is defined by excluding the regions near non-Lipschitz points or approximating the associated function values.
	We refer interested readers to \cite{Chen2021high,Chen2019complexity,Zhang2024riemannian} for further details.
	Our definition through the smoothing function serves a similar purpose.
	For example, the smoothing function based on \eqref{eq:tilde-abs} introduces a mild perturbation in the neighborhoods of non-Lipschitz points.
\end{remark}

\section{An Adaptive Smoothing Algorithm}

\label{sec:algorithm}

Leveraging the constructions in Section~\ref{sec:smoothing}, we now devise a smoothing algorithm to solve problem~\eqref{opt:rnlip}.
This approach proceeds by performing Riemannian gradient steps to minimize a smoothing surrogate of the objective function, where the stepsize is updated by an adaptive scheme.
In addition, a decreasing sequence for the smoothing parameter is prescribed in advance.
The integration of these ingredients results in a fully parameter-free algorithm, which does not require any prior knowledge of problem-specific parameters and yet ensures convergence guarantees.

\subsection{Basic Framework}

The smoothing function constructed in \eqref{eq:tilde-f} furnishes an effective approximation to the original objective function.
A natural strategy for solving problem~\eqref{opt:rnlip} is to employ the Riemannian gradient algorithm to minimize this smoothing function, while progressively reducing the smoothing parameter to refine the approximation.
Accordingly, we consider the following update procedure for the $k$-th iteration,
\begin{equation}
	\label{eq:srg}
	X_{k + 1} = \retr_{X_k} \dkh{ - \gamma_k \grad\, \tilde{f}_{\mu_k} (X_k) },
\end{equation}
where $\gamma_k > 0$ and $\mu_k > 0$ are the current stepsize and smoothing parameter, respectively.
The remaining challenge lies in determining the appropriate update rules for both the stepsize and the smoothing parameter.

In contrast to existing methods \cite{Wang2025distributionally,Zhang2024riemannian} that update the smoothing parameter only upon satisfying certain conditions, we assign the smoothing parameter at iteration $k$ directly as
\begin{equation*}
	\mu_k = k^{- 1 / (4 - p)}.
\end{equation*}
This configuration is inspired by the variable smoothing technique \cite{Bohm2021variable,Bot2015variable}.
We have tailored the decay rate of the smoothing parameter to accommodate the non-Lipschitz nature of problem~\eqref{opt:rnlip}, which plays a pivotal role in characterizing the iteration complexity.

A prevalent approach for stepsize selection relies on backtracking line search.
In the context of Riemannian optimization, such a procedure entails substantial computational overhead, as each backtracking step requires an additional retraction to project the iterate back onto the manifold (see \cite{Beck2023dynamic,Chen2020proximal,Zhang2024riemannian}).
Consequently, line-search methods often become computationally inefficient in this setting.
In light of this consideration, we instead turn our attention to AdaGrad-type algorithms \cite{Duchi2011adaptive,Gratton2024complexity}, which will be elaborated in the next subsection.

\subsection{Algorithm Design}

To better illustrate our proposed strategy for stepsizes, we begin by reviewing the AdaGrad algorithm in the Euclidean setting.
Let $\psi: \Rns \to \bR$ be a continuously differentiable function to be minimized.
By accumulating the squared norms of past gradients, AdaGrad adapts the stepsize based on the overall scale of the optimization landscape.
Specifically, AdaGrad updates the iterate according to
\begin{equation*}
	\left\{
	\begin{aligned}
		X_{k + 1} & = X_k - \dfrac{1}{\eta_k} \nabla \psi (X_k), \\
		\eta_{k + 1}^2 & = \eta_k^2 + \norm{\nabla \psi (X_{k + 1})}\fs,
	\end{aligned}
	\right.
\end{equation*}
with $X_0 \in \Rns$ and $\eta_0 > 0$.
We refer interested readers to \cite{Duchi2011adaptive,Gratton2024complexity} for further details about AdaGrad.

We may naturally incorporate the update scheme of AdaGrad into the iterative process \eqref{eq:srg}.
However, as indicated by Lemma~\ref{le:sgrad-lip}, the Lipschitz constant of $\nabla \tilde{f}_{\mu}$ scales as $O (\mu^{p - 2})$, a factor not accounted for in the original AdaGrad formulation.
To navigate this challenge, we devise a novel adaptive stepsize strategy in accordance with the current smoothing parameter.
The complete framework is summarized in Algorithm~\ref{alg:ASRGA}, which is named {\it Adaptive Smoothing Riemannian Gradient Algorithm} and abbreviated to ASRGA.
In contrast to the classical AdaGrad method, we accumulate the squared norm of the smoothing Riemannian gradient after scaling by the smoothing parameter in \eqref{eq:update-eta}.
Likewise, the stepsize itself is also scaled by the smoothing parameter in \eqref{eq:update-as}.
This carefully designed scaling mechanism constitutes an essential ingredient for our theoretical development.

\begin{algorithm2e}[ht]
	\caption{Adaptive Smoothing Riemannian Gradient Algorithm (ASRGA).}
	\label{alg:ASRGA}
	\KwIn{$X_{0} = X_{1} \in \cM$, $\mu_0 \in (0, 1]$, and $\eta_0 > 0$.}
	
	\For{$k = 1, 2, \dotsc$}{
		
		Compute $\mu_k = k^{- 1 / (4 - p)}$ and
		\begin{equation}
			\label{eq:update-eta}
			\eta_k^2 = \eta_{k - 1}^2 + \mu_k^{2 - p} \norm{\grad\, \tilde{f}_{\mu_k} (X_{k})}\fs.
		\end{equation}
		
		Update $X_{k + 1}$ by
		\begin{equation}
			\label{eq:update-as}
			X_{k + 1} = \retr_{X_{k}} \dkh{ - \dfrac{\mu_k^{2 - p}}{\eta_k} \grad\, \tilde{f}_{\mu_k} (X_{k}) }.
		\end{equation}
		
	}

\end{algorithm2e}

\begin{remark}
	A direct extension of AdaGrad to our setting can be expressed as
	\begin{equation}
		\label{eq:adagrad}
		\left\{
		\begin{aligned}
			X_{k + 1} & = \retr_{X_k} \dkh{ - \grad\, \tilde{f}_{\mu_k} (X_k) / \eta_k }, \\
			\eta_{k + 1}^2 & = \eta_k^2 + \norm{\grad\, \tilde{f}_{\mu_{k + 1}} (X_{k + 1})}\fs,
		\end{aligned}
		\right.
	\end{equation}
	with $X_0 \in \cM$ and $\eta_0 > 0$.
	Our empirical analysis demonstrates that the global convergence of the algorithmic framework \eqref{eq:adagrad} can be guaranteed with the iteration complexity $O (\epsilon^{2 (p - 4) / p})$.
	In the case where $p = 1$, this iteration complexity is limited to $O (\epsilon^{-6})$, considerably inferior to existing results \cite{Beck2023dynamic}.
\end{remark}

\section{Convergence Analysis}

\label{sec:convergence}

This section delves into a comprehensive convergence analysis of Algorithm~\ref{alg:ASRGA}, thereby corroborating the effectiveness of the adaptive stepsize strategy.
Specifically, we establish that any accumulation point of the generated sequence is a stationary point.
And the iteration complexity of Algorithm \ref{alg:ASRGA} is provided to attain an approximate stationary point.

The lemma below establishes a fundamental inequality concerning the cumulative squared norms of smoothing Riemannian gradients, serving as a cornerstone in the analysis that follows.

\begin{lemma}
	\label{le:sum-as}
	Let $\{X_k\}$ be the sequence generated by Algorithm \ref{alg:ASRGA}.
	Then, for any $k \in \bN$, it holds that
	\begin{equation}
		\label{eq:sum-as}
		\sum_{j = 1}^{k} \dfrac{\mu_j^{2 - p}}{\eta_j} \norm{\grad\, \tilde{f}_{\mu_j} (X_j)}\fs
		\leq \overline{f} - \underline{f} + C_f \mu_1^{p}
		+ \sum_{j = 1}^{k} \dfrac{\rho_j \mu_j^{2 (2 - p)}}{2 \eta_j^2} \norm{\grad\, \tilde{f}_{\mu_j} (X_j)}\fs.
	\end{equation}
	Here, $\overline{f}$ and $\underline{f}$ are two constants defined in Lemma~\ref{le:bound}, and $C_f = 2^{(1 - p) / (4 - p)} m \kappa \sigma^{p - 1}$ is a positive constant.
\end{lemma}

\begin{proof}
	Let $\rho_k = R_1^2 L_{\mu_k} + 2 R_2 M_{\mu_k}$.
	According to the updating rule \eqref{eq:update-as} of $X_{k + 1}$ and Corollary~\ref{coro:retr}, it follows that
	\begin{equation}
		\label{eq:des-f}
		\begin{aligned}
			\tilde{f}_{\mu_k} (X_{k + 1})
			= {} & \tilde{f}_{\mu_k} ( \retr_{X_k} ( - \mu_k^{2 - p} \grad\, \tilde{f}_{\mu_k} (X_k) / \eta_k ) ) \\
			\leq {} & \tilde{f}_{\mu_k} (X_k)
			- \dfrac{\mu_k^{2 - p}}{\eta_k} \norm{\grad\, \tilde{f}_{\mu_k} (X_k)}\fs
			+ \dfrac{\rho_k \mu_k^{2(2 - p)}}{2 \eta_k^2} \norm{\grad\, \tilde{f}_{\mu_k} (X_k)}\fs.
		\end{aligned}
	\end{equation}
	By invoking the results of Lemma~\ref{le:sgrad-mu}, we obtain that
	\begin{equation}
		\label{eq:des-mu}
		\begin{aligned}
			\tilde{f}_{\mu_{k}} (X_{k + 1})
			- \tilde{f}_{\mu_{k + 1}} (X_{k + 1})
			\geq {} & - m \kappa \sigma^{p - 1} \dkh{\dfrac{\mu_{k}}{\mu_{k + 1}}}^{1 - p} \dkh{\mu_{k}^{p} - \mu_{k + 1}^{p}} \\
			\geq {} & - 2^{(1 - p) / (4 - p)} m \kappa \sigma^{p - 1} \dkh{\mu_{k}^{p} - \mu_{k + 1}^{p}},
		\end{aligned}
	\end{equation}
	where the last inequality holds since $\mu_{k} / \mu_{k + 1} = ((k + 1) / k )^{1 / (4 - p)} \leq 2^{1 / (4 - p)}$.
	Collecting two relationships \eqref{eq:des-f} and \eqref{eq:des-mu} together yields that
	\begin{equation*}
		\begin{aligned}
			\dfrac{\mu_k^{2 - p}}{\eta_k} \norm{\grad\, \tilde{f}_{\mu_k} (X_k)}\fs
			\leq {} & \tilde{f}_{\mu_k} (X_k) - \tilde{f}_{\mu_{k + 1}} (X_{k + 1})
			+ C_f \dkh{\mu_{k}^{p} - \mu_{k + 1}^{p}} \\
			& + \dfrac{\rho_k \mu_k^{2(2 - p)}}{2 \eta_k^2} \norm{\grad\, \tilde{f}_{\mu_k} (X_k)}\fs.
		\end{aligned}
	\end{equation*}
	Then it can be readily verified that
	\begin{equation*}
		\begin{aligned}
			\sum_{j = 1}^{k} \dfrac{\mu_j^{2 - p}}{\eta_j} \norm{\grad\, \tilde{f}_{\mu_j} (X_j)}\fs
			\leq {} & \tilde{f}_{\mu_1} (X_1) - \tilde{f}_{\mu_{k + 1}} (X_{k + 1})
			+ C_f \dkh{\mu_1^{p} - \mu_{k + 1}^{p}} \\
			& + \sum_{j = 1}^{k} \dfrac{\rho_j \mu_j^{2 (2 - p)}}{2 \eta_j^2} \norm{\grad\, \tilde{f}_{\mu_j} (X_j)}\fs.
		\end{aligned}
	\end{equation*}
	Combining the above inequality with Lemma~\ref{le:bound} and the fact $\mu_{k + 1} \geq 0$, we arrive at the assertion \eqref{eq:sum-as} of this lemma and finish the proof.
\end{proof}

We need the following technical results in the subsequent analysis.

\begin{lemma}
	\label{le:adagrad}
	The following two statements hold.
	\begin{enumerate}[(i)]
		
		\item Let $\{a_k\}$ be a nonnegative sequence and $\xi > 0$ be a constant. 
		Then we have 
		\begin{equation*}
			\sum_{j = 1}^{k} \dfrac{a_j}{\xi + \sum_{i = 1}^{j} a_i} \leq \log \dkh{ 1 + \dfrac{1}{\xi} \sum_{j = 1}^{k} a_j }.
		\end{equation*}
		
		\item Let $a$, $b$, and $c$ be three nonnegative constants. If $a t \leq b + c t$ for $t > 1$, we have
		\begin{equation*}
			t \leq \max \hkh{ e^{b / c}, \dfrac{4c^2}{a^2} }.
		\end{equation*}
		
	\end{enumerate}
\end{lemma}

\begin{proof}
	The proof of these two inequalities can be found in \cite[Lemma~3.1]{Gratton2024complexity} and \cite[Lemma~3.6]{Gratton2025simple}, respectively, which is omitted here. 
\end{proof}

To analyze the convergence of the sequence $\{X_k\}$ generated by Algorithm~\ref{alg:ASRGA}, we define an auxiliary sequence $\{\tilde{X}_{k} = X_{N_k}\}$ and $\{\tilde{\mu}_k = \mu_{N_k}\}$, where
\begin{equation*}
	N_k \in \argmin_{j \in \{\lfloor k / 2\rfloor, \dotsc, k\}} \norm{ \grad\, \tilde{f}_{\mu_j} (X_{j}) }\ff.
\end{equation*}
Rather than taking the most recent point, $\tilde{X}_k$ is chosen from a sliding window of past iterates as the one achieving the smallest norm of the smoothing Riemannian gradient.
Drawing on the above results, we move on to characterize the rate at which the smoothing Riemannian gradient decays.

\begin{proposition}
	Let $\{\tilde{X}_{k}\}$ be the sequence generated by Algorithm \ref{alg:ASRGA}.
	Then, for all $k \in \bN$, we have
	\begin{equation}
		\label{eq:grad-limit}
		\norm{\grad\, \tilde{f}_{\tilde{\mu}_k} (\tilde{X}_{k})}\ff
		\leq \dfrac{\eta_0 Q}{k^{1 / (4 - p)}},
	\end{equation}
	where $Q > 0$ is a constant defined as
	\begin{equation*}
		Q = \max \hkh{ \sqrt{2}, \; \dfrac{16 C_g^2}{\eta_0^2}, \; \exp \dkh{\dfrac{\overline{f} - \underline{f} + C_f \mu_1^{p}}{C_g}} },
	\end{equation*}
	with $C_g = R_1^2 (L_g \mu_1^{2 - p} + L_1 L_{h} \mu_1 + L_2) + 2 R_2 (M_g \mu_1^{2 - p} + L_1 M_h \mu_1)$.
\end{proposition}

\begin{proof}
	Let
	\begin{equation*}
		S_k = \sum_{j = 1}^{k} \mu_j^{2 - p} \norm{\grad\, \tilde{f}_{\mu_j} (X_j)}\fs.
	\end{equation*}
	We first consider the scenario where $S_k < \eta_0^2$.
	Then it follows from the relationship $\mu_j \geq \mu_k$ for all $j = 1, 2, \dotsc, k$ that
	\begin{equation*}
		\begin{aligned}
			\eta_0^2
			\geq {} & \mu_k^{2 - p} \sum_{j = 1}^{k} \norm{\grad\, \tilde{f}_{\mu_j} (X_j)}\fs
			\geq \mu_k^{2 - p} \sum_{j = \lfloor k / 2 \rfloor}^{k} \norm{\grad\, \tilde{f}_{\mu_j} (X_j)}\fs \\
			\geq {} & \dfrac{\mu_k^{2 - p} k}{2}\norm{\grad\, \tilde{f}_{\tilde{\mu}_k} (\tilde{X}_k)}\fs
			= \dfrac{k^{2 / (4 - p)}}{2} \norm{\grad\, \tilde{f}_{\tilde{\mu}_k} (\tilde{X}_k)}\fs,
		\end{aligned}
	\end{equation*}
	where the last inequality holds due to the definition of $N_k$ and the last equality comes from the definition of $\mu_k$.
	Consequently, we can obtain that
	\begin{equation*}
		\norm{\grad\, \tilde{f}_{\tilde{\mu}_k} (\tilde{X}_k)}\ff
		\leq \dfrac{\sqrt{2} \eta_0}{k^{1 / (4 - p)}}
		\leq \dfrac{\eta_0 Q}{k^{1 / (4 - p)}}.
	\end{equation*}
	Next, our focus is shifted to the situation where $S_k \geq \eta_0^2$.
	Then it holds that
	\begin{equation*}
		\eta_k^2
		= \eta_0^2 + \sum_{j = 1}^{k} \mu_j^{2 - p} \norm{\grad\, \tilde{f}_{\mu_j} (X_j)}\fs
		= \eta_0^2 + S_k
		\leq 2 S_k.
	\end{equation*}
	Since $\eta_j \leq \eta_k$ for all $j = 1, 2, \dotsc, k$, we have
	\begin{equation}
		\label{eq:sum-lhs}
		\sum_{j = 1}^{k} \dfrac{\mu_j^{2 - p}}{\eta_j} \norm{\grad\, \tilde{f}_{\mu_j} (X_j)}\fs
		\geq \dfrac{1}{\eta_k} \sum_{j = 1}^{k} \mu_j^{2 - p} \norm{\grad\, \tilde{f}_{\mu_j} (X_j)}\fs
		= \dfrac{S_k}{\eta_k}
		\geq \dfrac{\sqrt{S_k}}{\sqrt{2}}.
	\end{equation}
	Moreover, it follows from the relationship $\rho_j = R_1^2 L_{\mu_j} + 2 R_2 M_{\mu_j} \leq C_g \mu_j^{p - 2}$ that
	\begin{equation*}
		\sum_{j = 1}^{k} \dfrac{\rho_j \mu_j^{2 (2 - p)}}{2 \eta_j^2} \norm{\grad\, \tilde{f}_{\mu_j} (X_j)}\fs
		\leq \dfrac{C_g}{2} \sum_{j = 1}^{k} \dfrac{\mu_j^{2 - p}}{\eta_j^2} \norm{\grad\, \tilde{f}_{\mu_j} (X_j)}\fs.
	\end{equation*}
	By invoking the results of Lemma~\ref{le:adagrad} (i), we obtain that
	\begin{equation*}
		\begin{aligned}
			\sum_{j = 1}^{k} \dfrac{\mu_j^{2 - p}}{\eta_j^2} \norm{\grad\, \tilde{f}_{\mu_j} (X_j)}\fs
			= {} & \sum_{j = 1}^{k} \dfrac{\mu_j^{2 - p} \norm{\grad\, \tilde{f}_{\mu_j} (X_j)}\fs }{\eta_0^2 + \sum_{i = 1}^{j} \mu_i^{2 - p} \norm{\grad\, \tilde{f}_{\mu_i} (X_i)}\fs} \\
			\leq {} & \log \dkh{1 + \dfrac{1}{\eta_0^2} \sum_{j = 1}^{k} \mu_j^{2 - p} \norm{\grad\, \tilde{f}_{\mu_j} (X_j)}\fs} \\
			= {} & \log \dkh{ 1 + \dfrac{S_k}{\eta_0^2} },
		\end{aligned}
	\end{equation*}
	which together with the condition $S_k \geq \eta_0^2$ implies that
	\begin{equation*}
		\sum_{j = 1}^{k} \dfrac{\mu_j^{2 - p}}{\eta_j^2} \norm{\grad\, \tilde{f}_{\mu_j} (X_j)}\fs
		\leq \log \dkh{\dfrac{2 S_k}{\eta_0^2}}.
	\end{equation*}
	Hence, it is straightforward to verify that
	\begin{equation}
		\label{eq:sum-rhs}
		\sum_{j = 1}^{k} \dfrac{\rho_j \mu_j^{2 (2 - p)}}{2 \eta_j^2} \norm{\grad\, \tilde{f}_{\mu_j} (X_j)}\fs
		\leq \dfrac{C_g}{2} \log \dkh{\dfrac{2 S_k}{\eta_0^2}}.
	\end{equation}
	Collecting three relationships \eqref{eq:sum-as}, \eqref{eq:sum-lhs} and \eqref{eq:sum-rhs} together, we then arrive at
	\begin{equation*}
		\dfrac{\sqrt{S_k}}{\sqrt{2}}
		\leq \overline{f} - \underline{f} + C_f \mu_1^{p}
		+ \dfrac{C_g}{2} \log \dkh{\dfrac{2 S_k}{\eta_0^2}},
	\end{equation*}
	which, after a suitable rearrangement, can be equivalently written as
	\begin{equation*}
		\dfrac{\eta_0}{2} \sqrt{ \dfrac{2 S_k}{\eta_0^2} }
		\leq \overline{f} - \underline{f} + C_f \mu_1^{p}
		+ C_g \log \dkh{ \sqrt{\dfrac{2 S_k}{\eta_0^2}} }.
	\end{equation*}
	According to Lemma~\ref{le:adagrad} (ii), it follows that
	\begin{equation*}
		\sqrt{ \dfrac{2 S_k}{\eta_0^2} }
		\leq \max \hkh{ \dfrac{16 C_g^2}{\eta_0^2}, \; \exp \dkh{\dfrac{\overline{f} - \underline{f} + C_f \mu_1^{p}}{C_g}} }
		\leq Q.
	\end{equation*}
	As a result, we have
	\begin{equation*}
		\begin{aligned}
			\dfrac{\eta_0^2 Q^2}{2}
			\geq {} & S_k
			= \sum_{j = 1}^{k} \mu_j^{2 - p} \norm{\grad\, \tilde{f}_{\mu_j} (X_j)}\fs
			\geq \mu_k^{2 - p} \sum_{j = 1}^{k} \norm{\grad\, \tilde{f}_{\mu_j} (X_j)}\fs \\
			\geq {} & \mu_k^{2 - p} \sum_{j = \lfloor k / 2 \rfloor}^{k} \norm{\grad\, \tilde{f}_{\mu_j} (X_j)}\fs
			\geq \dfrac{k^{2 / (4 - p)}}{2} \norm{\grad\, \tilde{f}_{\tilde{\mu}_k} (\tilde{X}_k)}\fs,
		\end{aligned}
	\end{equation*}
	which indicates that
	\begin{equation*}
		\norm{\grad\, \tilde{f}_{\tilde{\mu}_k} (\tilde{X}_k)}\ff
		\leq \dfrac{\eta_0 Q}{k^{1 / (4 - p)}}.
	\end{equation*}
	The proof is completed.
\end{proof}

We are now in a position to establish the global convergence and iteration complexity of Algorithm~\ref{alg:ASRGA}, as articulated in the theorem below.

\begin{theorem}
	\label{thm:asrga}
	Let $\{\tilde{X}_{k}\}$ be the sequence generated by Algorithm~\ref{alg:ASRGA}.
	Then $\{\tilde{X}_{k}\}$ has at least one accumulation point and any accumulation point qualifies as a stationary point of problem~\eqref{opt:rnlip}.
	Moreover, Algorithm~\ref{alg:ASRGA} will reach an $\epsilon$-approximate stationary point of problem~\eqref{opt:rnlip} after at most $O (\epsilon^{p - 4})$ iterations.
\end{theorem}

\begin{proof}
	Since $\cM$ is a compact manifold, the sequence $\{\tilde{X}_k\}$ generated by Algorithm~\ref{alg:ASRGA} is bounded.
	Then from the Bolzano-Weierstrass theorem, it can be deduced that this sequence has at least one accumulation point.
	Let $X\uast$ be an accumulation point of $\{\tilde{X}_k\}$.
	For notational simplicity, we continue to denote by $\{\tilde{X}_k\}$ the subsequence converging to $X\uast$.
	The completeness of $\cM$ guarantees that $X\uast \in \cM$.
	Upon taking $k \to \infty$ in \eqref{eq:grad-limit}, we immediately arrive at
	\begin{equation*}
		\lim_{k \to \infty} \norm{\grad\, \tilde{f}_{\tilde{\mu}_k} (\tilde{X}_{k})}\ff = 0.
	\end{equation*}
	Moreover, since $N_k \geq \lfloor k / 2 \rfloor$, it holds that
	\begin{equation}
		\label{eq:mu-limit}
		\tilde{\mu}_k
		= \mu_{N_k}
		\leq \mu_{\lfloor k / 2 \rfloor}
		= \dfrac{1}{\lfloor k / 2 \rfloor^{1 / (4 - p)}},
	\end{equation}
	which indicates that $\lim_{k \to \infty} \tilde{\mu}_k  = 0$.
	According to Definition~\ref{def:rsubd}, we conclude that $0 \in \tilde{\partial}\lsfR f (X\uast)$.
	Consequently, the accumulation point $X\uast$ is a stationary point of problem \eqref{opt:rnlip}.
	Now we set
	\begin{equation*}
		k_{\epsilon}\uast = \left\lceil \max \hkh{ \dfrac{(\eta_0 Q)^{4 - p}}{\epsilon^{4 - p}}, \; \dfrac{2}{\epsilon^{4 - p}} + 2 } \right\rceil.
	\end{equation*}
	Then it directly follows from \eqref{eq:grad-limit} and \eqref{eq:mu-limit} that
	\begin{equation*}
		\norm{\grad\, \tilde{f}_{\tilde{\mu}_{k_{\epsilon}\uast}} (\tilde{X}_{k_{\epsilon}\uast})}\ff \leq \epsilon,
		\mbox{~~and~~}
		\tilde{\mu}_{k_{\epsilon}\uast} \leq \epsilon,
	\end{equation*}
	respectively.
	Therefore, the iterate $\tilde{X}_{k_{\epsilon}\uast}$ is an $\epsilon$-approximate stationary point of problem~\eqref{opt:rnlip}.
	We complete the proof.
\end{proof}

Theorem~\ref{thm:asrga} demonstrates that the iteration complexity of Algorithm~\ref{alg:ASRGA} is $O(\epsilon^{p - 4})$ for achieving an $\epsilon$-approximate stationary point of problem~\eqref{opt:rnlip}, which matches existing results in the Euclidean space \cite{Liu2016smoothing}.
For the special case of $p = 1$, corresponding to the classical regime of locally Lipschitz continuous functions, this complexity bound reduces to $O(\epsilon^{- 3})$, consistent with previously established findings on manifolds \cite{Beck2023dynamic}.
To the best of our knowledge, this is the first result regarding the iteration complexity for non-Lipschitz Riemannian optimization problems.

\section{Numerical Results}

\label{sec:numerical}

Preliminary numerical results are presented in this section to validate the effectiveness and efficiency of the proposed algorithm.
All codes are implemented in MATLAB R2018b on a workstation with dual Intel Xeon Gold 6242R CPU processors (at $3.10$ GHz$\times 20 \times 2$) and $510$ GB of RAM under Ubuntu 20.04.
In the subsequent experiments, the retraction operators are realized on the basis of polar decomposition.

\subsection{Comparison on SDL Problems}

We first engage in a numerical comparison between ASRGA and RSSD \cite{Zhang2024riemannian} on the SDL problem~\eqref{opt:sdl}.
Both algorithms construct the smoothing function based on the formulation given in \eqref{eq:tilde-abs}.
What distinguishes RSSD from ASRGA is that the former determines the stepsize at each iteration through a line-search procedure.
The corresponding experiments are configured following the procedures described in \cite{Li2021weakly,Zhang2024riemannian}.
More specifically, we begin by generating a random matrix whose entries are drawn from the standard normal distribution.
This matrix is then orthonormalized to yield $X\uast \in \Onn$, which serves as the underlying orthogonal dictionary.
Next, we randomly generate a sparse coefficient matrix $S \in \Rnm$ such that the entries follow a Bernoulli-Gaussian distribution with parameter $0.5$.
The data matrix $Y$ in problem~\eqref{opt:sdl} is finally constructed as $Y = X\uast S$.

In our experiments, we set the parameter $\eta_0$ to be $10^{-6}$ in ASRGA, while the parameters of RSSD are selected in accordance with the specification prescribed in \cite{Zhang2024riemannian}.
Starting from the same initial point that is randomly generated, both algorithms are employed to solve problem~\eqref{opt:sdl} with $n = 50$ and $m = \lfloor 10 n^{1.5} \rfloor = 3535$.
The termination criterion is imposed by a maximum runtime of $4$ seconds.
To assess the performance of the tested algorithms, we define the error between $X$ and $X\uast$ as
\begin{equation*}
	\mathrm{err} (X, X\uast) = \sum_{i = 1}^{n} \abs{ \max_{i = 1, \dotsc, n} \abs{X_i\zz X\uast_j} - 1},
\end{equation*}
where $X_i$ is the $i$-th column of $X$ and $X\uast_j$ is the $j$-th column of $X\uast$.
It is clear that $\mathrm{err} (X, X\uast) = 0$ if and only if $X$ and $X\uast$ are equal up to permutation and sign ambiguities.
Figure~\ref{fig:sdl} depicts the decay of errors over time for two algorithms with $p \in \{0.2, 0.5, 0.8\}$.
As evidenced therein, ASRGA exhibits a clear advantage in terms of computational efficiency, attaining a substantially higher level of accuracy within the same time budget.

\begin{figure}[t]
	\centering
	\subfigure[$p = 0.2$]{
		\label{subfig:p2}
		\includegraphics[width=0.32\linewidth]{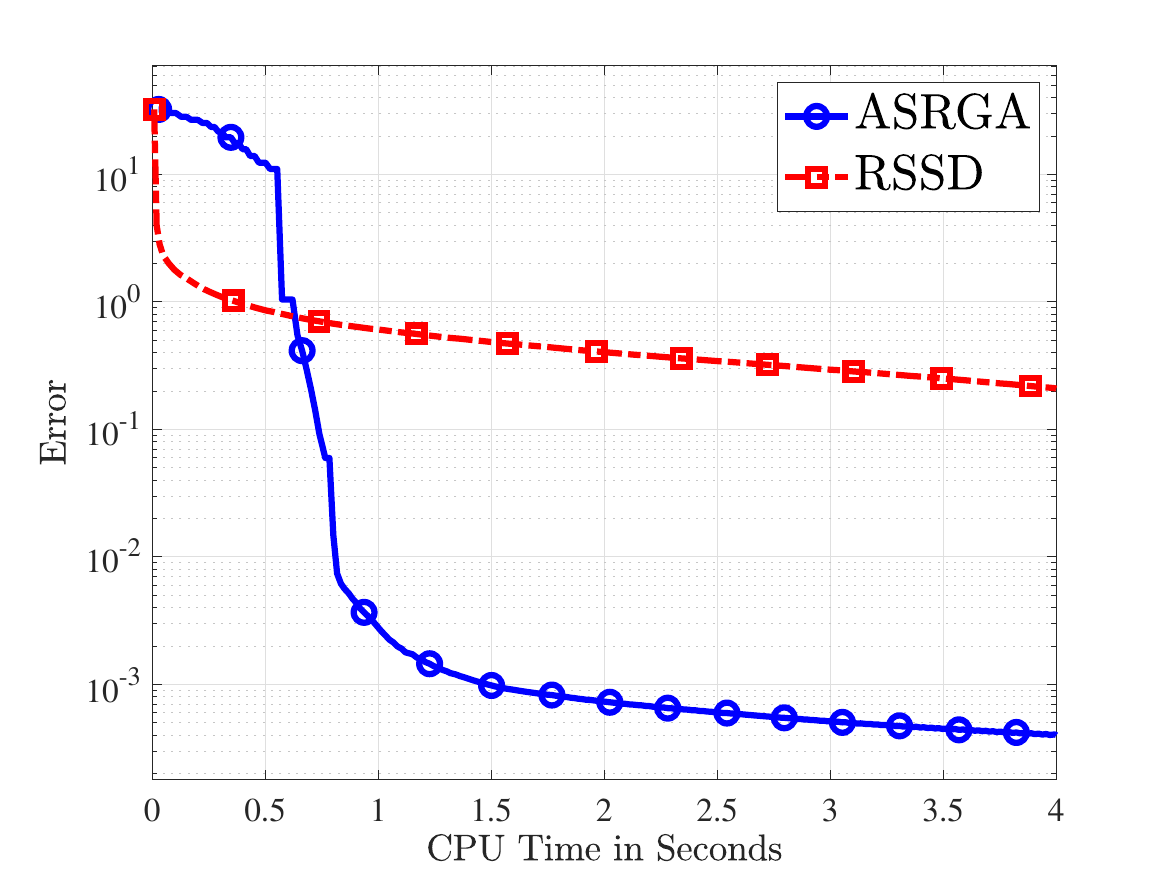}
	}\hspace{-10pt}%
	\subfigure[$p = 0.5$]{
		\label{subfig:p5}
		\includegraphics[width=0.32\linewidth]{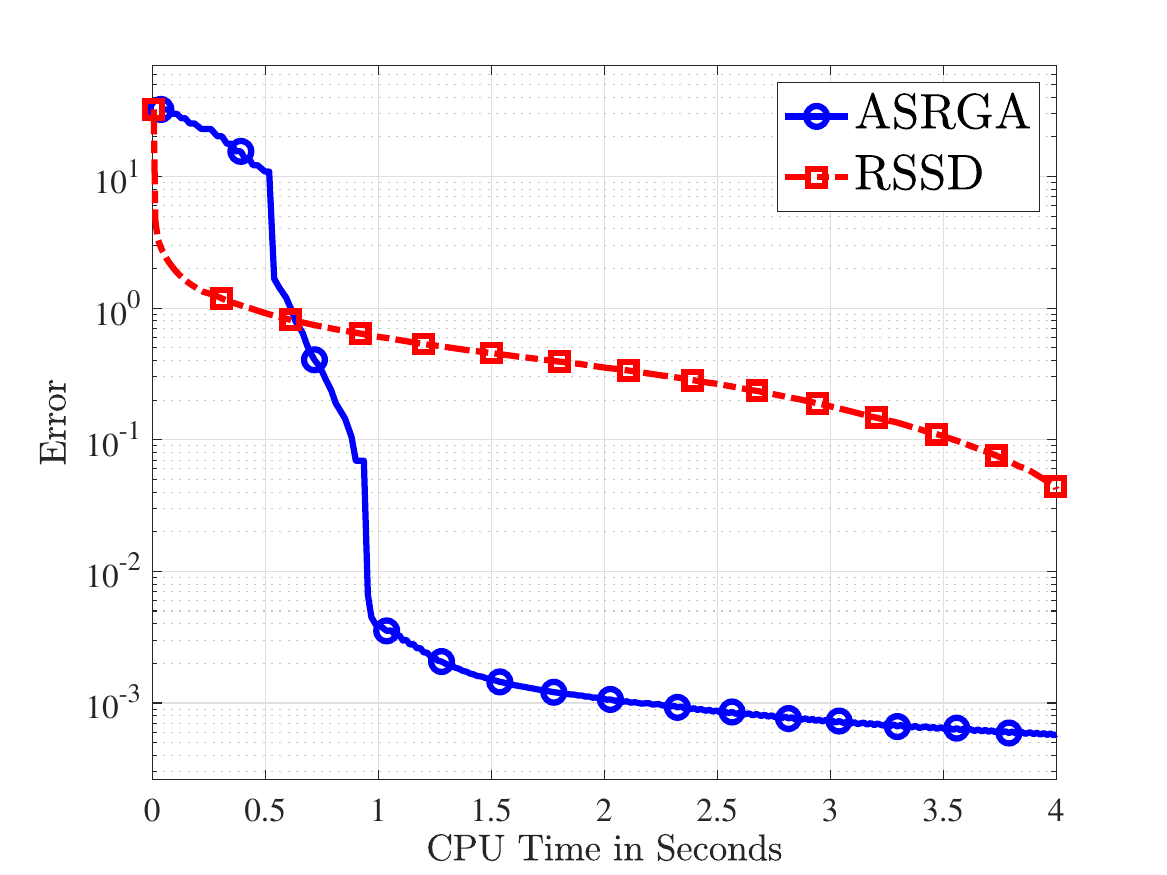}
	}\hspace{-10pt}%
	\subfigure[$p = 0.8$]{
		\label{subfig:p8}
		\includegraphics[width=0.32\linewidth]{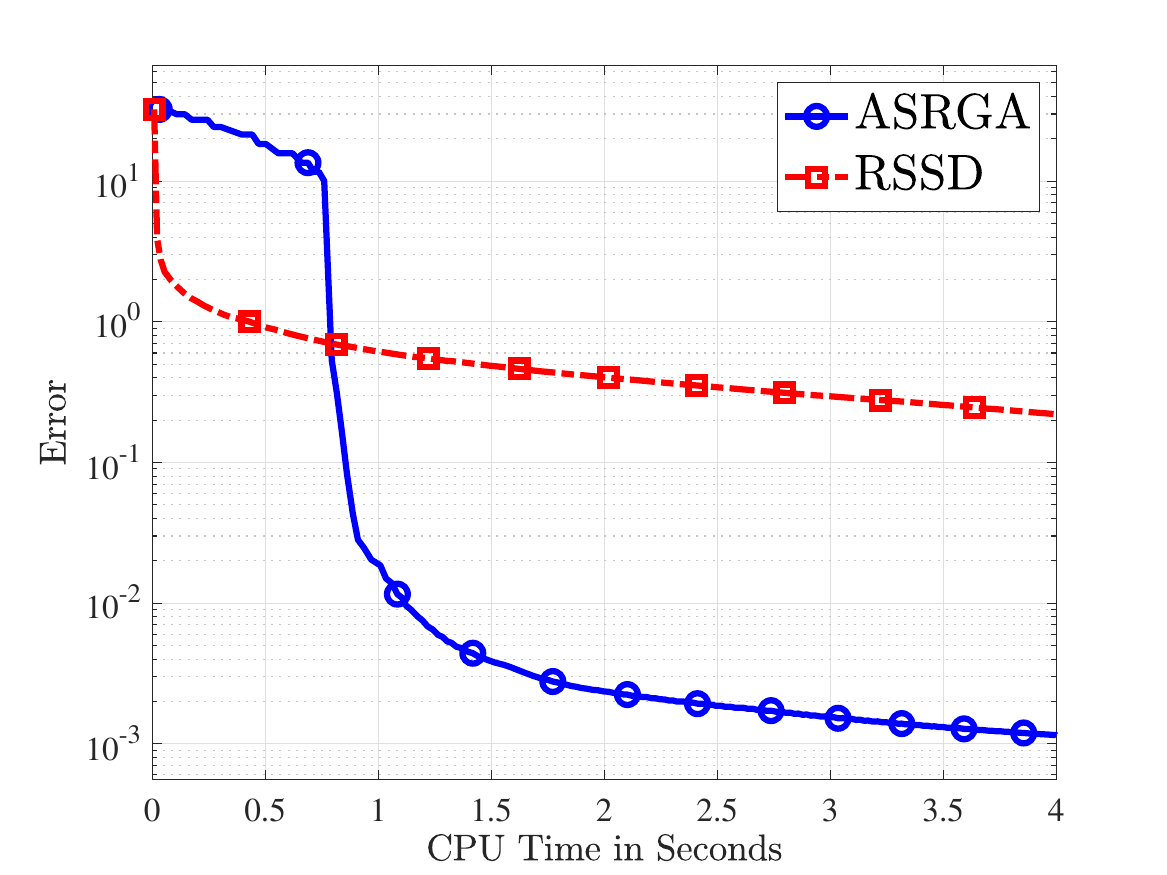}
	}
	\caption{Numerical comparison between ASRGA and RSSD for the SDL problem~\eqref{opt:sdl}.}
	\label{fig:sdl}
\end{figure}

\subsection{Comparison on DPCP Problems}

Our next experiment focuses on a fundamental application in autonomous driving, namely, the road detection problem.
We conduct a comparative evaluation of ASRGA and RSSD \cite{Zhang2024riemannian} on the KITTI dataset \cite{Geiger2013vision}, which is recorded from a moving platform during urban driving in and around Karlsruhe, Germany.
This dataset provides synchronized camera images and the corresponding 3D point clouds collected by a rotating laser scanner.
For each 3D point cloud of a road scene, we aim to identify the points that lie on the road plane (inliers) and those deviating from it (outliers).
Typically, each 3D point cloud contains more than $10^5$ points, with outliers accounting for over $50\%$.
By adopting homogeneous coordinates, this task can be cast as a DPCP problem~\eqref{opt:dpcp} with $n = 4$.
We refer interested readers to \cite{Ding2019noisy,Tsakiris2018dual} for further details on the problem setup and its background.

The parameter $\eta_0$ in ASRGA is set to $10^{-6}$, while the parameters of RSSD are chosen according to the settings in \cite{Zhang2024riemannian}.
Both algorithms are run under a fixed time budget of $0.5$ seconds with the same initial point.
We define the detection error as
\begin{equation*}
	\sqrt{ 1 - \abs{ \jkh{ x_{\mathrm{alg}}, x_{\mathrm{opt}} } }^2 },
\end{equation*}
where $x_{\mathrm{alg}}$ denotes the solution returned by a tested algorithm and $x_{\mathrm{opt}}$ represents the ground-truth solution.
Table~\ref{tb:KITTI} tabulates the final errors achieved by two algorithms together with the required numbers of iterations on six frames from the KITTI dataset.
As evidenced therein, ASRGA consistently attains smaller errors than RSSD under the same time budget, demonstrating its superior efficiency.

\begin{table}[ht]
	\small
	\centering
	\caption{Comparison between ASRGA and RSSD on KITTI with $0.5$ seconds of runtime (\#Inliers: number of inliers, \#Outliers: number of outliers).}
	\label{tb:KITTI}
	\begin{tabular}{
			>{\centering\arraybackslash}p{0.20\textwidth}
			>{\centering\arraybackslash}p{0.08\textwidth}
			>{\centering\arraybackslash}p{0.10\textwidth}
			>{\centering\arraybackslash}p{0.12\textwidth}
			>{\centering\arraybackslash}p{0.12\textwidth}
			>{\centering\arraybackslash}p{0.12\textwidth}
		}
		\toprule
		Dataset & Frame & \#Inliers & \#Outliers & Algorithm & Error \\
		\midrule
		\multirow{4}{*}{\rule{0pt}{3.5ex}KITTI-CITY-5}
		& \multirow{2}{*}{$120$}
		& \multirow{2}{*}{$54461$}
		& \multirow{2}{*}{$63489$}
		& ASRGA & $3.8 \times 10^{-3}$ \\
		& & & & RSSD & $1.7 \times 10^{-2}$ \\
		\cmidrule{2-6}
		& \multirow{2}{*}{$153$}
		& \multirow{2}{*}{$31845$}
		& \multirow{2}{*}{$65042$}
		& ASRGA & $4.2 \times 10^{-3}$ \\
		& & & & RSSD & $3.7 \times 10^{-2}$ \\
		\midrule
		\multirow{4}{*}{\rule{0pt}{3.5ex}KITTI-CITY-48}
		& \multirow{2}{*}{$0$}
		& \multirow{2}{*}{$51096$}
		& \multirow{2}{*}{$66280$}
		& ASRGA & $1.9 \times 10^{-3}$ \\
		& & & & RSSD & $1.1 \times 10^{-2}$ \\
		\cmidrule{2-6}
		& \multirow{2}{*}{$21$}
		& \multirow{2}{*}{$50309$}
		& \multirow{2}{*}{$68749$}
		& ASRGA & $8.7 \times 10^{-4}$ \\
		& & & & RSSD & $1.8 \times 10^{-2}$ \\
		\midrule
		\multirow{4}{*}{\rule{0pt}{3.5ex}KITTI-CITY-71}
		& \multirow{2}{*}{$328$}
		& \multirow{2}{*}{$44631$}
		& \multirow{2}{*}{$77135$}
		& ASRGA & $9.1 \times 10^{-3}$ \\
		& & & & RSSD & $2.3 \times 10^{-2}$ \\
		\cmidrule{2-6}
		& \multirow{2}{*}{$881$}
		& \multirow{2}{*}{$42965$}
		& \multirow{2}{*}{$80643$}
		& ASRGA & $4.0 \times 10^{-4}$ \\
		& & & & RSSD & $7.3 \times 10^{-2}$ \\
		\bottomrule
	\end{tabular}
\end{table}

Finally, Figure~\ref{fig:KITTI} showcases qualitative detection results on two representative frames from the KITTI dataset with inliers in blue and outliers in red, which are visualized by projecting the 3D point clouds onto the images.
We can observe that, ASRGA succeeds in distinguishing the road plane from surrounding objects, whereas RSSD fails to produce a reliable separation.
These results provide compelling evidence of the practical effectiveness and robustness of ASRGA in real-world autonomous driving scenarios.

\begin{figure}[t]
	\centering
	\subfigure[Frame 153 of KITTI-CITY-5]{
		\label{subfig:26_005_0153}
		\includegraphics[width=0.45\linewidth]{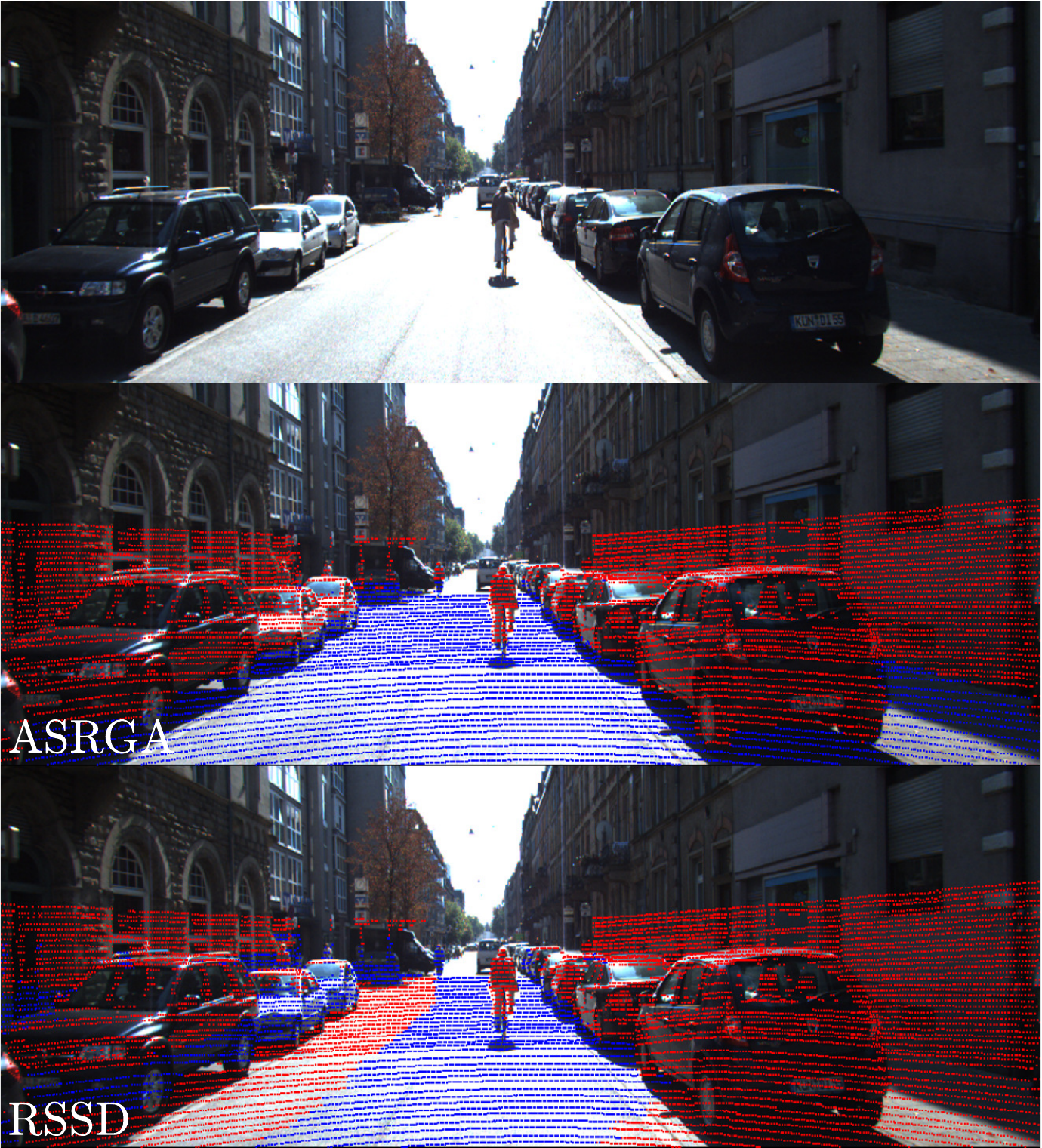}
	}
	\subfigure[Frame 881 of KITTI-CITY-71]{
		\label{subfig:29_071_0881}
		\includegraphics[width=0.45\linewidth]{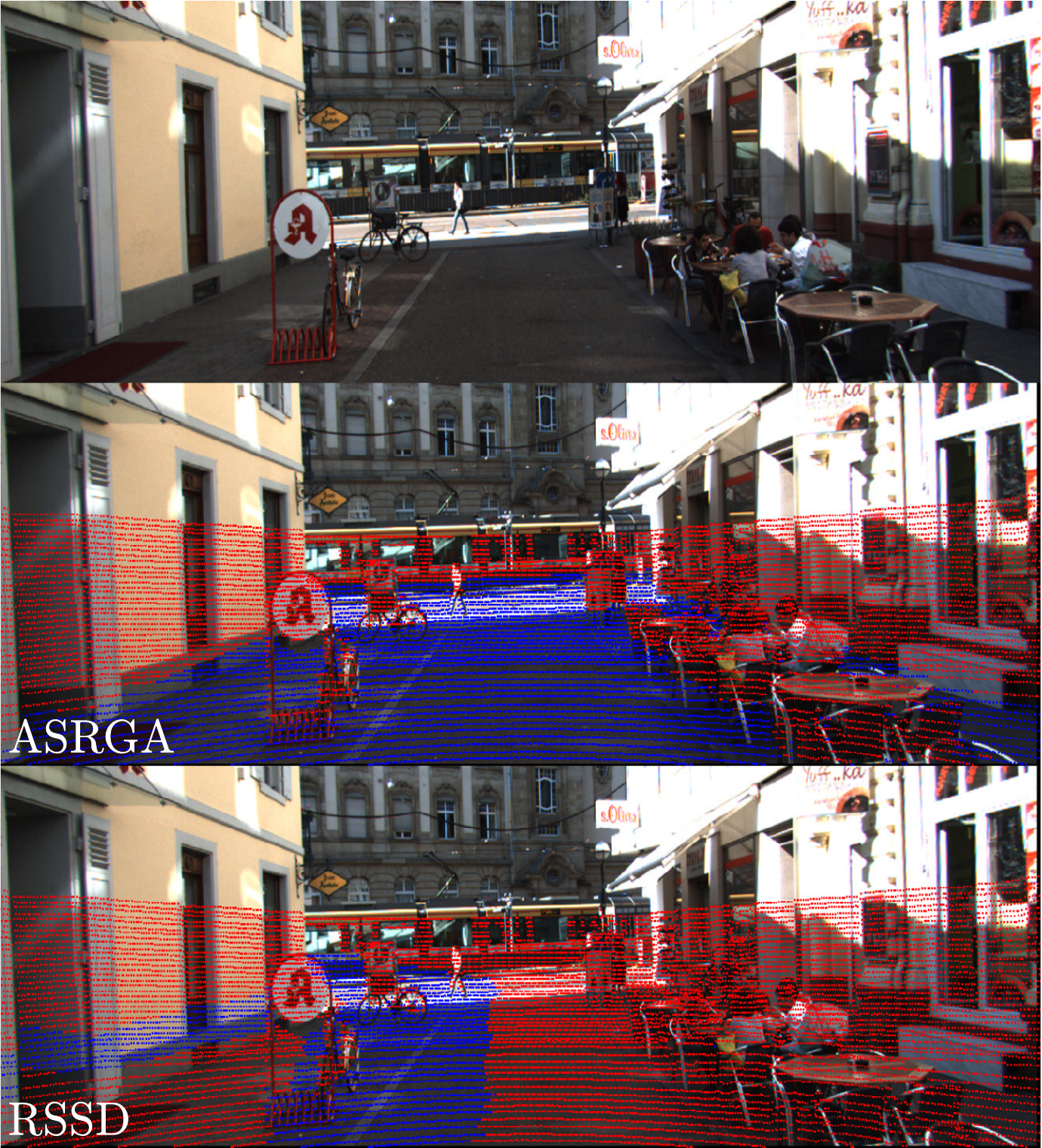}
	}
	\caption{Illustration of two frames from the KITTI dataset.
		The first row corresponds to the raw images.
		The second and third rows display the detection results produced by ASRGA and RSSD, respectively, with inliers in blue and outliers in red.}
	\label{fig:KITTI}
\end{figure}

\subsection{Comparison on SPCA Problems}

In the final stage of numerical experiments, our attention is directed toward evaluating the performance of ASRGA on the SPCA problem~\eqref{opt:spca}.
Since problems of this kind are characterized by $p = 1$, we accordingly choose ManPG \cite{Chen2020proximal} and DSGM \cite{Beck2023dynamic} as benchmark algorithms for a comparative testing.
Both ManPG and DSGM hinge on a line-search procedure for stepsize selection.
Their implementations are sourced from GitHub\footnote{See \url{https://github.com/chenshixiang/ManPG} for ManPG and \url{https://github.com/israelross/DSGM/} for DSGM.}.
We retain the original settings and configure the parameters in accordance with the specifications provided in \cite{Beck2023dynamic,Chen2020proximal}.
The parameter $\eta_0$ in ASRGA is set to $10^{-6}$ in the following experiments.

We draw upon the widely recognized image dataset MNIST \cite{LeCun1998gradient} in the realm of machine learning research, from which the first $m = 10000$ samples with $n = 784$ features are extracted to construct $B \in \Rnm$.
The matrix $A$ in problem~\eqref{opt:spca} is generated by subtracting the sample-mean from each sample of $B$ and then normalizing each row of the resulting matrix to unit norm.
For our testing, we fix the coefficient of the $\ell_1$-norm regularizer as $\lambda = 0.5$, while varying $s$ across the values in $\{30, 40, 50\}$.
All algorithms are initialized from the same starting point that is randomly generated and allowed to run for $5$ seconds.
We adopt the objective function value of problem~\eqref{opt:spca} as the performance metric.

Figure~\ref{fig:spca} illustrates the decline in function values for the tested algorithms with respect to CPU time, with the three subplots corresponding, respectively, to the three different values of $s$.
As shown, ASRGA outperforms both ManPG and DSGM, which reaches the lowest function value within the same allotted time.
Moreover, the superiority of ASRGA becomes more pronounced as the value of $s$ grows.
This phenomenon can be attributed to the fact that larger problem scales entail a higher computational cost for an evaluation of retractions.
Consequently, the line-search procedure becomes increasingly time-consuming, thereby undermining the overall efficiency of ManPG and DSGM.
These results highlight the advantage of the adaptive stepsize strategy employed by ASRGA.

\begin{figure}[t]
	\centering
	\subfigure[$s = 30$]{
		\label{subfig:s30}
		\includegraphics[width=0.32\linewidth]{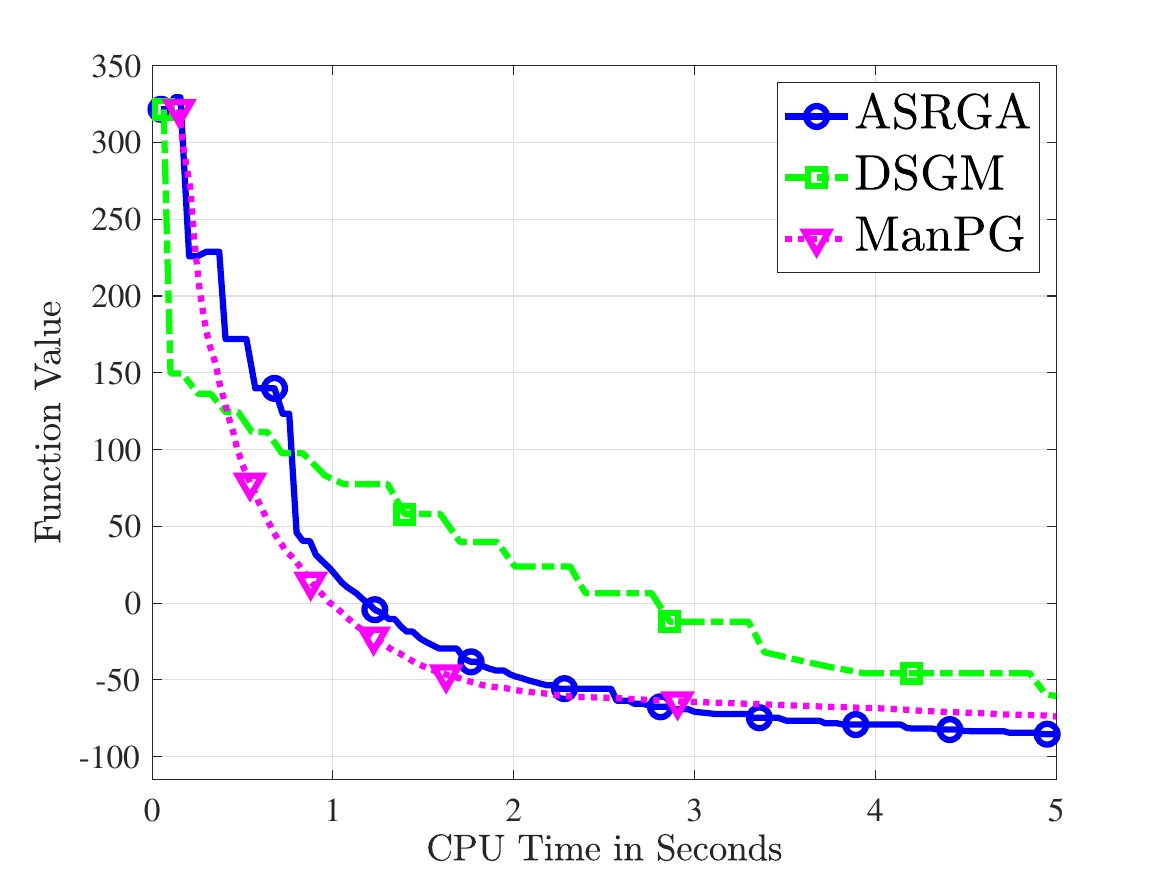}
	}\hspace{-10pt}%
	\subfigure[$s = 40$]{
		\label{subfig:s40}
		\includegraphics[width=0.32\linewidth]{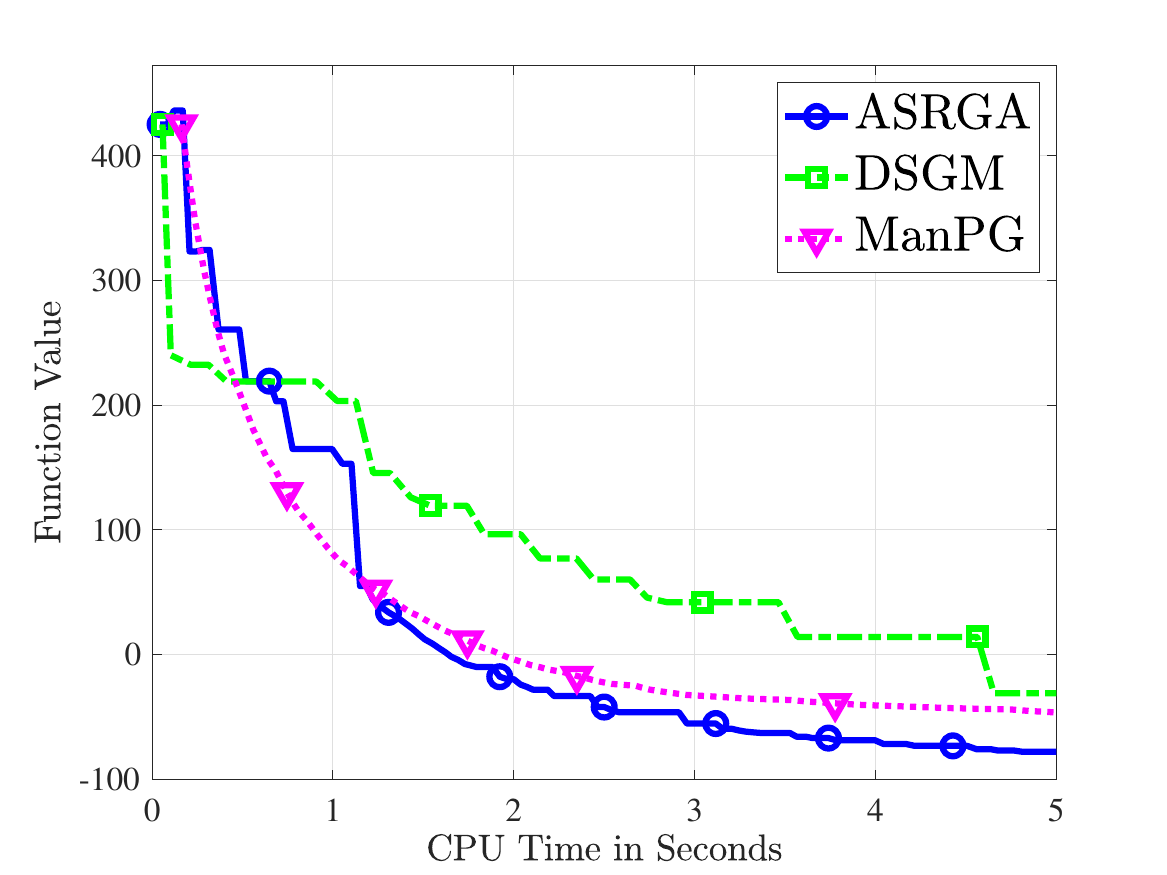}
	}\hspace{-10pt}%
	\subfigure[$s = 50$]{
		\label{subfig:s50}
		\includegraphics[width=0.32\linewidth]{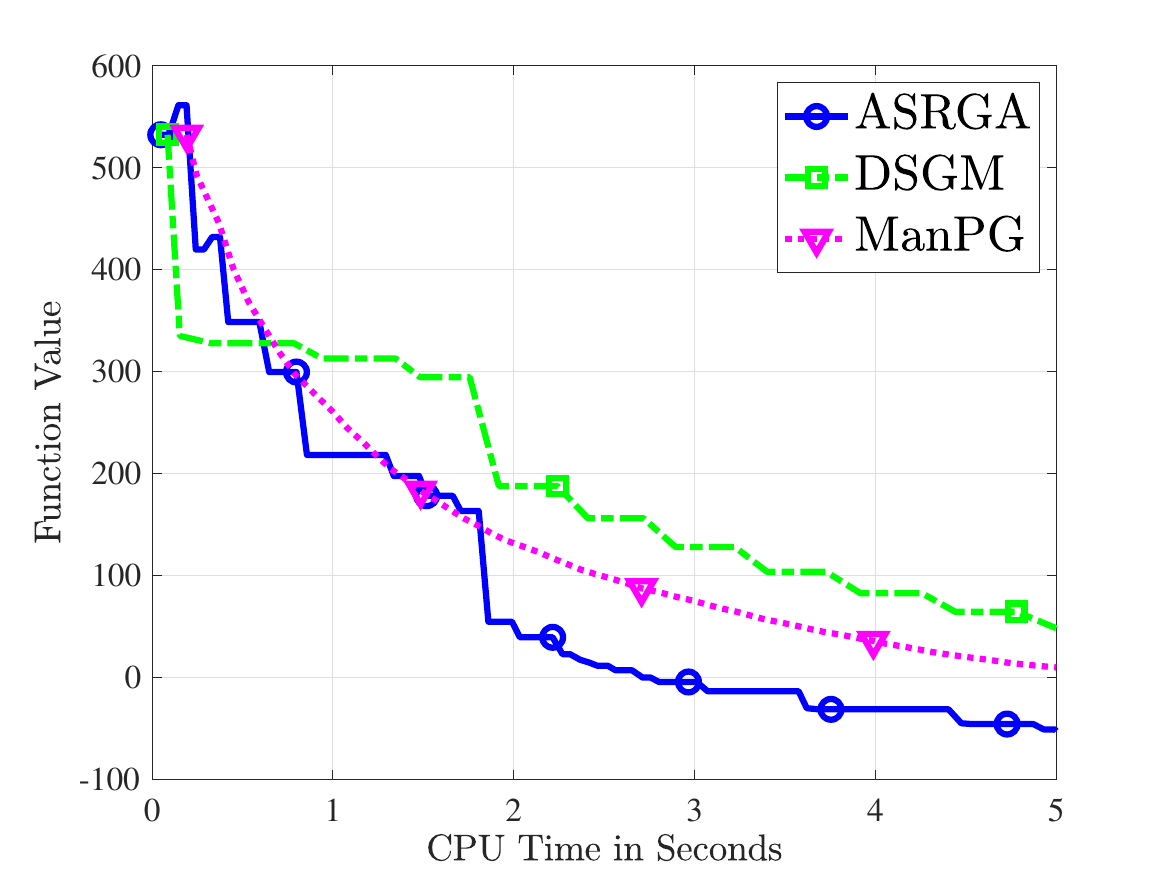}
	}
	\caption{Numerical comparison among ASRGA, DSGM, and ManPG for the SPCA problem~\eqref{opt:spca}.}
	\label{fig:spca}
\end{figure}

\section{Concluding Remarks} 

\label{sec:conclusion}

We develop an efficient and easy-to-implement algorithm ASRGA for a class of non-Lipschitz Riemannian optimization problems of the form \eqref{opt:rnlip}.
Our theoretical analysis demonstrates that an $\epsilon$-approximate stationary point of problem~\eqref{opt:rnlip} can be attained by ASRGA in $O (\epsilon^{p - 4})$ iterations.
In particular, ASRGA enjoys this complexity guarantee while remaining entirely parameter-free, which makes it especially appealing in applications where the Lipschitz constants of the involved smoothing function are unavailable or difficult to estimate.
Preliminary numerical results further validate that the proposed algorithm is not only theoretically sound, but also computationally competitive in practice.
Our algorithm outperforms existing methods not only for non-Lipschitz problems but also for Lipschitz problems.

Several directions deserve further study.
It would be of considerable interest to extend the present framework to stochastic settings for large-scale problems in machine learning and to explore higher-order complexity bounds for a broader class of problems.
The techniques developed herein may serve as a useful step toward a systematic complexity theory for non-Lipschitz optimization on Riemannian manifolds.


\bibliographystyle{abbrv}
\bibliography{library}

\begin{thebibliography}{10}

\bibitem{Absil2008optimization}
P.-A. Absil, R.~Mahony, and R.~Sepulchre.
\newblock {\em Optimization Algorithms on Matrix Manifolds}.
\newblock Princeton University Press, Princeton, 2008.

\bibitem{Beck2023dynamic}
A.~Beck and I.~Rosset.
\newblock A dynamic smoothing technique for a class of nonsmooth optimization
  problems on manifolds.
\newblock {\em SIAM Journal on Optimization}, 33(3):1473--1493, 2023.

\bibitem{Bian2013worst}
W.~Bian and X.~Chen.
\newblock Worst-case complexity of smoothing quadratic regularization methods
  for non-{L}ipschitzian optimization.
\newblock {\em SIAM Journal on Optimization}, 23(3):1718--1741, 2013.

\bibitem{Bian2015complexity}
W.~Bian, X.~Chen, and Y.~Ye.
\newblock Complexity analysis of interior point algorithms for non-{L}ipschitz
  and nonconvex minimization.
\newblock {\em Mathematical Programming}, 149(1):301--327, 2015.

\bibitem{Bohm2021variable}
A.~B{\"o}hm and S.~J. Wright.
\newblock Variable smoothing for weakly convex composite functions.
\newblock {\em Journal of Optimization Theory and Applications}, 188:628--649,
  2021.

\bibitem{Bot2015variable}
R.~I. Bo{\c{t}} and C.~Hendrich.
\newblock A variable smoothing algorithm for solving convex optimization
  problems.
\newblock {\em TOP}, 23:124--150, 2015.

\bibitem{Boumal2023introduction}
N.~Boumal.
\newblock {\em An Introduction to Optimization on Smooth Manifolds}.
\newblock Cambridge University Press, Cambridge, 2023.

\bibitem{Boumal2018global}
N.~Boumal, P.-A. Absil, and C.~Cartis.
\newblock Global rates of convergence for nonconvex optimization on manifolds.
\newblock {\em IMA Journal of Numerical Analysis}, 39(1):1--33, 2018.

\bibitem{Chen2020proximal}
S.~Chen, S.~Ma, A.~M.-C. So, and T.~Zhang.
\newblock Proximal gradient method for nonsmooth optimization over the
  {S}tiefel manifold.
\newblock {\em SIAM Journal on Optimization}, 30(1):210--239, 2020.

\bibitem{Chen2024nonsmooth}
S.~Chen, S.~Ma, A.~M.-C. So, and T.~Zhang.
\newblock Nonsmooth optimization over the {S}tiefel manifold and beyond:
  Proximal gradient method and recent variants.
\newblock {\em SIAM Review}, 66(2):319--352, 2024.

\bibitem{Chen2016augmented}
W.~Chen, H.~Ji, and Y.~You.
\newblock An augmented {L}agrangian method for $\ell_{1}$-regularized
  optimization problems with orthogonality constraints.
\newblock {\em SIAM Journal on Scientific Computing}, 38(4):B570--B592, 2016.

\bibitem{Chen2012smoothing}
X.~Chen.
\newblock Smoothing methods for nonsmooth, nonconvex minimization.
\newblock {\em Mathematical Programming}, 134:71--99, 2012.

\bibitem{Chen2025tight}
X.~Chen, Y.~He, and Z.~Zhang.
\newblock Tight error bounds for the sign-constrained {S}tiefel manifold.
\newblock {\em SIAM Journal on Optimization}, 35(1):302--329, 2025.

\bibitem{Chen2026orthogonal}
X.~Chen, W.~Li, and Q.~Luo.
\newblock Orthogonal nonnegative matrix factorization via minimization over the
  null space.
\newblock {\em to appear in SIAM Journal on Matrix Analysis and Applications},
  2026.

\bibitem{Chen2021high}
X.~Chen and P.~L. Toint.
\newblock High-order evaluation complexity for convexly-constrained
  optimization with non-{L}ipschitzian group sparsity terms.
\newblock {\em Mathematical Programming}, 187(1):47--78, 2021.

\bibitem{Chen2019complexity}
X.~Chen, P.~L. Toint, and H.~Wang.
\newblock Complexity of partially separable convexly constrained optimization
  with non-{L}ipschitzian singularities.
\newblock {\em SIAM Journal on Optimization}, 29(1):874--903, 2019.

\bibitem{Chen2010lower}
X.~Chen, F.~Xu, and Y.~Ye.
\newblock Lower bound theory of nonzero entries in solutions of
  $\ell_2$-$\ell_p$ minimization.
\newblock {\em SIAM Journal on Scientific Computing}, 32(5):2832--2852, 2010.

\bibitem{Deng2025oracle}
K.~Deng, J.~Hu, J.~Wu, and Z.~Wen.
\newblock Oracle complexities of augmented {L}agrangian methods for nonsmooth
  composite optimization on a compact submanifold.
\newblock {\em Mathematics of Operations Research}, pages 1--27, 2025.

\bibitem{Ding2019noisy}
T.~Ding, Z.~Zhu, T.~Ding, Y.~Yang, D.~P. Robinson, M.~C. Tsakiris, and
  R.~Vidal.
\newblock Noisy dual principal component pursuit.
\newblock In {\em Proceedings of the 36th International Conference on Machine
  Learning}, volume~97, pages 1617--1625. PMLR, 2019.

\bibitem{Duchi2011adaptive}
J.~Duchi, E.~Hazan, and Y.~Singer.
\newblock Adaptive subgradient methods for online learning and stochastic
  optimization.
\newblock {\em Journal of Machine Learning Research}, 12(61):2121--2159, 2011.

\bibitem{Garmanjani2013smoothing}
R.~Garmanjani and L.~N. Vicente.
\newblock Smoothing and worst-case complexity for direct-search methods in
  nonsmooth optimization.
\newblock {\em IMA Journal of Numerical Analysis}, 33(3):1008--1028, 2013.

\bibitem{Geiger2013vision}
A.~Geiger, P.~Lenz, C.~Stiller, and R.~Urtasun.
\newblock Vision meets robotics: The {KITTI} dataset.
\newblock {\em The International Journal of Robotics Research},
  32(11):1231--1237, 2013.

\bibitem{Gratton2024complexity}
S.~Gratton, S.~Jerad, and P.~L. Toint.
\newblock Complexity of a class of first-order objective-function-free
  optimization algorithms.
\newblock {\em Optimization Methods and Software}, pages 1--31, 2024.

\bibitem{Gratton2025simple}
S.~Gratton and P.~L. Toint.
\newblock A simple first-order algorithm for full-rank equality constrained
  optimization.
\newblock {\em arXiv:2510.16390}, 2025.

\bibitem{Huang2022riemannian}
W.~Huang and K.~Wei.
\newblock {R}iemannian proximal gradient methods.
\newblock {\em Mathematical Programming}, 194(1-2):371--413, 2022.

\bibitem{Huang2023inexact}
W.~Huang and K.~Wei.
\newblock An inexact {R}iemannian proximal gradient method.
\newblock {\em Computational Optimization and Applications}, 85(1):1--32, 2023.

\bibitem{Jiang2023exact}
B.~Jiang, X.~Meng, Z.~Wen, and X.~Chen.
\newblock An exact penalty approach for optimization with nonnegative
  orthogonality constraints.
\newblock {\em Mathematical Programming}, 198(1):855--897, 2023.

\bibitem{Jiang2025inexact}
B.~Jiang, M.~Xu, X.~Cai, and Y.-F. Liu.
\newblock An inexact proximal framework for nonsmooth {R}iemannian
  difference-of-convex optimization.
\newblock {\em arXiv:2509.08561}, 2025.

\bibitem{LeCun1998gradient}
Y.~LeCun, L.~Bottou, Y.~Bengio, and P.~Haffner.
\newblock Gradient-based learning applied to document recognition.
\newblock {\em Proceedings of the IEEE}, 86(11):2278--2324, 1998.

\bibitem{Li2025riemannian}
J.~Li, S.~Ma, and T.~Srivastava.
\newblock A {R}iemannian alternating direction method of multipliers.
\newblock {\em Mathematics of Operations Research}, 50(4):3222--3242, 2025.

\bibitem{Li2021weakly}
X.~Li, S.~Chen, Z.~Deng, Q.~Qu, Z.~Zhu, and A.~M.-C. So.
\newblock Weakly convex optimization over {S}tiefel manifold using {R}iemannian
  subgradient-type methods.
\newblock {\em SIAM Journal on Optimization}, 31(3):1605--1634, 2021.

\bibitem{Li2026unsupervised}
Y.~Li, D.~Sun, and L.~Zhang.
\newblock Unsupervised feature selection via nonnegative orthogonal constrained
  regularized minimization.
\newblock {\em Journal of Machine Learning Research}, 27(39):1--44, 2026.

\bibitem{Liu2016smoothing}
Y.-F. Liu, S.~Ma, Y.-H. Dai, and S.~Zhang.
\newblock A smoothing {SQP} framework for a class of composite {$L_q$}
  minimization over polyhedron.
\newblock {\em Mathematical Programming}, 158(1):467--500, 2016.

\bibitem{Peng2023riemannian}
Z.~Peng, W.~Wu, J.~Hu, and K.~Deng.
\newblock {R}iemannian smoothing gradient type algorithms for nonsmooth
  optimization problem on compact {R}iemannian submanifold embedded in
  {E}uclidean space.
\newblock {\em Applied Mathematics $\&$ Optimization}, 88(3):85, 2023.

\bibitem{Qian2024error}
Y.~Qian, S.~Pan, and L.~Xiao.
\newblock Error bound and exact penalty method for optimization problems with
  nonnegative orthogonal constraint.
\newblock {\em IMA Journal of Numerical Analysis}, 44(1):120--156, 2024.

\bibitem{Tsakiris2018dual}
M.~C. Tsakiris and R.~Vidal.
\newblock Dual principal component pursuit.
\newblock {\em Journal of Machine Learning Research}, 19(18):1--49, 2018.

\bibitem{Wang2025decentralized}
L.~Wang, L.~Bao, and X.~Liu.
\newblock A decentralized proximal gradient tracking algorithm for composite
  optimization on {R}iemannian manifolds.
\newblock {\em Journal of Machine Learning Research}, 26(106):1--37, 2025.

\bibitem{Wang2022decentralized}
L.~Wang and X.~Liu.
\newblock Decentralized optimization over the {S}tiefel manifold by an
  approximate augmented {L}agrangian function.
\newblock {\em IEEE Transactions on Signal Processing}, 70:3029--3041, 2022.

\bibitem{Wang2023smoothing}
L.~Wang and X.~Liu.
\newblock Smoothing gradient tracking for decentralized optimization over the
  {S}tiefel manifold with non-smooth regularizers.
\newblock In {\em Proceedings of the 62nd IEEE Conference on Decision and
  Control (CDC)}, pages 126--132. IEEE, 2023.

\bibitem{Wang2025distributionally}
L.~Wang, X.~Liu, and X.~Chen.
\newblock The distributionally robust optimization model of sparse principal
  component analysis.
\newblock {\em arXiv:2503.02494}, 2025.

\bibitem{Wang2025support}
L.~Wang, X.~Liu, and X.~Chen.
\newblock A support-set algorithm for optimization problems with nonnegative
  and orthogonal constraints.
\newblock {\em arXiv:2511.03443}, 2025.

\bibitem{Wang2024seeking}
L.~Wang, X.~Liu, and Y.~Zhang.
\newblock Seeking consensus on subspaces in federated principal component
  analysis.
\newblock {\em Journal of Optimization Theory and Applications}, 203:529--561,
  2024.

\bibitem{Zhai2020complete}
Y.~Zhai, Z.~Yang, Z.~Liao, J.~Wright, and Y.~Ma.
\newblock Complete dictionary learning via $\ell_4$-norm maximization over the
  orthogonal group.
\newblock {\em Journal of Machine Learning Research}, 21(165):1--68, 2020.

\bibitem{Zhang2024riemannian}
C.~Zhang, X.~Chen, and S.~Ma.
\newblock A {R}iemannian smoothing steepest descent method for non-{L}ipschitz
  optimization on embedded submanifolds of $\mathbb{R}^n$.
\newblock {\em Mathematics of Operations Research}, 49(3):1710--1733, 2024.

\bibitem{Zheng2025new}
Z.~Zheng, X.~Yu, S.~Ma, and L.~Xue.
\newblock A new inexact manifold proximal linear algorithm with adaptive
  stopping criteria.
\newblock {\em arXiv:2508.19234}, 2025.

\end{thebibliography}

\addcontentsline{toc}{section}{References}

\end{document}